\documentclass[twoside, 11pt, leqno]{amsart}
 \usepackage[latin1]{inputenc}
 \usepackage[T1]{fontenc}
 \usepackage{amsmath, amssymb, amsfonts, amsthm, amscd}
 \usepackage[left=120pt, right=120pt, bottom=110pt, top=100pt, footskip=22pt]{geometry}
 \usepackage{mathrsfs}
 \usepackage{graphicx}
 \usepackage[all]{xy}
 \usepackage{array, booktabs, ragged2e}
 \usepackage{url}
 \usepackage{slashed}

\theoremstyle{plain}
\numberwithin{equation}{section}

\newtheorem{theorem}[equation]{Theorem}
\newtheorem{maintheorem}{Theorem}
\newtheorem{corollary}[equation]{Corollary}
\newtheorem{lemma}[equation]{Lemma}
\newtheorem{proposition}[equation]{Proposition}
\newtheorem{conjecture}[equation]{Conjecture}
\newtheorem{definition}[equation]{Definition}
\newtheorem{definitions}[equation]{Definitions}
\newtheorem{remark}[equation]{Remark}
\newtheorem{remarks}[equation]{Remarks}

\newtheorem{analassump}[equation]{Analytic assumptions}
\newtheorem{naturassump}[equation]{Naturality assumptions}

\DeclareMathOperator{\modulo}{mod}

\DeclareMathOperator{\tr}{tr}

\DeclareMathOperator{\Hess}{Hess}
\DeclareMathOperator{\spec}{spec}
\DeclareMathOperator{\vol}{vol}
\DeclareMathOperator{\Id}{Id}
\DeclareMathOperator{\diverg}{div}
\DeclareMathOperator{\conf}{conf_{g_0}}
\DeclareMathOperator{\Diff}{Diff}
\DeclareMathOperator{\diff}{diff_{g_0}}
\DeclareMathOperator{\confdiff}{(conf+diff)_{g_0}}
\DeclareMathOperator{\confdiffperp}{(conf+diff)_{g_0}^\perp}
\DeclareMathOperator{\diffperp}{diff_{g_0}^\perp}
\DeclareMathOperator{\diffspin}{diff^{Spin}_{g_0}}
\DeclareMathOperator{\confdiffspin}{conf_{g_0}+diff^{Spin}_{g_0}}
\DeclareMathOperator{\projdiffperp}{\Pi_{diff^\perp}}
\DeclareMathOperator{\projconfdiffperp}{\Pi_{(conf+diff)^\perp}}

\newcommand{\fracsm}[2]{\begin{matrix}\frac{#1}{#2}\end{matrix}}
\newcommand{\prefix}[3]{\vphantom{#3}#1#2#3}
\newcommand{\beq}{\begin{equation}}
\newcommand{\eeq}{\end{equation}}

\newcommand{\Naturals}{\mathbb{N}}
\newcommand{\Reals}{\mathbb{R}}
\newcommand{\Complex}{\mathbb{C}}

\newcommand{\Real}{\mathrm{Re}\,}

\newcommand{\dalpha}{\partial^\alpha}   
\newcommand{\dbeta}{\partial^\beta}
\newcommand{\dgamma}{\partial^\gamma}
\newcommand{\ddelta}{\partial^\delta}
\newcommand{\dsubi}{\partial_i}
\newcommand{\dsubj}{\partial_j}\newcommand{\xij}{\xi_j}
\newcommand{\dsubk}{\partial_k}\newcommand{\xik}{\xi_k}
\newcommand{\dsubl}{\partial_l}
\newcommand{\dsubp}{\partial_p}\newcommand{\xip}{\xi_p}
\newcommand{\dsubq}{\partial_q}\newcommand{\xiq}{\xi_q}

\newcommand{\deltajk}{\delta_{jk}}
\newcommand{\deltaij}{\delta_{ij}}
\newcommand{\deltaik}{\delta_{ik}}
\newcommand{\deltajl}{\delta_{jl}}
\newcommand{\deltalj}{\delta_{jl}}
\newcommand{\deltakl}{\delta_{kl}}
\newcommand{\deltail}{\delta_{il}}
\newcommand{\xinorm}[1]{|\xi|^{#1}}  
\newcommand{\ijab}{^{ij}_{\alpha,\beta}}
\newcommand{\klgd}{^{kl}_{\gamma,\delta}}
\newcommand{\Cliff}{\mathrm{Cl}}          

\newcommand{\CliffC}{\Cliff^\Complex}
\newcommand{\PSpin}{\mathcal{P}_{\mathrm{Spin}}}
\newcommand{\Pin}[1]{\mathrm{Pin}(#1)}
\newcommand{\Spin}[1]{\mathrm{Spin}(#1)}

\newcommand{\Gl}[1]{\mathrm{Gl}(#1)}
\newcommand{\Oof}[1]{\mathrm{O}(#1)}
\newcommand{\SO}[1]{\mathrm{SO}(#1)}
\newcommand{\Dirac}{\slashed{\nabla}}
\newcommand{\Diracpresuffix}[2]{\prefix^{#1}{\Dirac}^{#2}}
\newcommand{\Cinf}[1]{C^\infty(#1)}          
\newcommand{\CinfStwoM}{C^\infty(S^2M)}
\newcommand{\CoTM}{T^*M}
\newcommand{\Endo}[1]{\mathrm{End}(#1)} 
\newcommand{\DIP}[2]{\langle\!\langle #1,#2\rangle\!\rangle}
\newcommand{\BigDIP}[2]{\Big\langle\!\!\Big\langle #1,#2\Big\rangle\!\!\Big\rangle}
\newcommand{\Metr}[1]{\mathrm{Metr}(#1)}

\makeindex

\begin{document}
\title[Extremal metrics for Dirac operators]{Extremal metrics for spectral functions of Dirac operators in even and odd dimensions}

\author{Niels Martin M\o{}ller}
\address{Niels Martin M\o{}ller, Department of Mathematics, Massachusetts Institute of Technology, Cambridge, MA 02139, USA.}
\email{moller@math.mit.edu}
\urladdr{http://math.mit.edu/~moller}
\thanks{This work was carried out at University of Aarhus, Denmark, and during visits at U. Penn. and Princeton University. The author was supported partially by his Elite Research Travel Grant 2007, Denmark, and would like to thank B. \O{}rsted for suggesting to study the determinant of the Dirac operator, as well as K. Okikiolu and several others for useful discussions.}
\date{August 2009}
\subjclass[2000]{58J52, 58J40}

\begin{abstract}
Let $(M^n,g)$ be a closed smooth Riemannian spin manifold and denote by $\Dirac$ its Atiyah-Singer-Dirac operator. We study the variation of Riemannian metrics for the zeta function and functional determinant of $\Dirac^2$, and prove finiteness of the Morse index at stationary metrics, and local extremality at such metrics under general, i.e. not only conformal, change of metrics.

In even dimensions, which is also a new case for the conformal Laplacian, the relevant stability operator is of log-polyhomogeneous pseudodifferential type, and we prove new results of independent interest, on the spectrum for such operators. We use this to prove local extremality under variation of the Riemannian metric, which in the important example when $(M^n,g)$ is the round $n$-sphere, gives a partial verification of Branson's conjecture on the pattern of extremals. Thus $\det\Dirac^2$ has a local (max, max, min, min) when the dimension is $(4k,4k+1,4k+2,4k+3)$, respectively.
\end{abstract}

\maketitle
\section{Introduction}
Fixing a closed smooth manifold $M$, the determinant of a natural geometric elliptic partial differential operator gives a
functional $g\mapsto\det P_g$ on the (infinite-dimensional) manifold of smooth Riemannian metrics
\[
\Metr{M}=\left\{g\in\Cinf{S^2TM}\:|\:g\:\textrm{is pos. def.}\right\},
\]
i.e. smooth symmetric positive definite 2-tensor fields on $M$, given by a certain regularization (i.e. renormalization) of the otherwise divergent product of the eigenvalues $\lambda_k$, where $|\lambda_0|\leq|\lambda_1|\leq\ldots\nearrow+\infty$ of the operator $P=P_g$ in each particular Riemannian metric $g$. More precisely, one studies the functional on the quotient
\beq\label{DetModuli}
\xymatrix {
F_P:\Metr{M}\Big/\mathrm{Diff}(M)\ar[r]&\Reals
},\quad F_P([g])=\det P_{[g]}
\eeq
The significance of this quotient is the identification of Riemannian metrics under the action of diffeomorphisms $\varphi:M\to M$ by
pullback $\varphi^*g$ of the metric $g$, since the spectrum of $P_g$ is invariant under this operation. Often the functional has additional invariance properties, e.g. the conformal invariance in many examples in theoretical physics, with differential operators such as the conformal Laplacian $L$ acting on scalar fields, and the Atiyah-Singer-Dirac operator $\Dirac$ on spinors. In quantum field and string theories, the determinant plays the rôle of an effective action (see \cite{DC}, \cite{Haw}) and hence the criticality and extremality properties of $\det P$ are both natural and of prime interest, and critical manifolds have interpretations in terms of the spacetime tracks of the evolution of strings and higher-dimensional objects.

From a purely geometric viewpoint, it is a fact that preciously little is currently known about the structure of the space of Riemannian metrics on a given smooth manifold $M$, and a key motivation here is the perspective use of spectral invariants for such investigations. Namely, in many previously considered cases the determinant is known to be extremal (within the conformal class) at special ``preferred'' metrics, such as constant curvature metrics, Einstein metrics, conformally flat metrics etc., depending on what the particular geometry allows. This was the case for the spheres in \cite{BrFuncDet}, \cite{Ok3} and \cite{On}, 2-dimensional surfaces in \cite{OPS1}, \cite{OPS2} and 4-manifolds in \cite{CY1}, the 4-dimensional boundary problem in \cite{CQ1} and \cite{CQ2}, where in each case the standard metric extremalizes the determinant of the natural conformally covariant operators considered there, viz. the conformal Laplacian and the Dirac operator. 

Note also, in connection with geometric flows for finding canonical metrics, that the gradient flow of the determinant of the Laplacian for 2-dimensional surfaces is equivalent to the Ricci flow (\cite{OPS1}),  and indeed leads to a proof of the classical uniformization theorem for Riemann surfaces (see e.g. \cite{CLT}). Functionals related to the determinant also played a rôle in Perelman's work (see Appendix B in \cite{CLN} for a detailed exposition, and see also \cite{MT}), and a related flow can be utilized to find hyperbolic metrics on knot complements.

It is important to recognize that the determinant in general is a very complicated global object, since its definition involves all the eigenvalues of the operator in question, and furthermore an analytic continuation procedure (or other essentially equivalent regularization technique). For instance for Riemann surfaces $(\Sigma^2,g)$, there are formulae for the determinant of the Laplacian in terms of the lengths of closed geodesics in $\Sigma^2$ (see \cite{Fr}, \cite{HP}, \cite{Sa} and \cite{PR}), illustrating the complicated global geometric nature of the invariant.

The well-known, yet remarkable, fact is that the variation under changes of metric of the determinant (in a precise sense) is a much more locally computable expression. This principle is the cornerstone of most of what is known about the determinant. For an arbitrary conformal variation, there exists in fact formulae for how the determinant changes, in terms of an integral of local invariants (i.e. curvatures and its covariant derivatives). These are the famous Polyakov-type formulas (e.g. \cite{Po}), which were historically utilized in the work by Osgood-Philips-Sarnak (e.g. \cite{OPS1}, \cite{OPS2}) on extremals of the determinant for 2-dimensional surfaces, and subsequently in the later results in these directions in \cite{BrFuncDet}, \cite{BO2}, \cite{CY1}, \cite{CQ1},\cite{CQ2}, in dimension 6 or lower. See also \cite{SZ} for recent results on perturbations of the sphere through singular metrics.

The approach with Polyakov formulas relies on explicit knowledge of heat invariants in terms of curvatures. To then obtain extremal results one applies sharp Beckner, Moser-Trudinger and Hardy-Littlewood-Sobolev-type inequalities to Polyakov formulas, after regrouping the curvature terms. From T. P. Branson's paper \cite{BrFuncDet} we have the following
theorem.
\begin{theorem}[Branson, 1993]\label{BransonsTheorem}
On $S^6$, for $g=e^{2\omega}g_{\rm{can}}$ in the conformal class of the
standard metric $g_{\rm{can}}$ having the standard volume, the determinant of
the Yamabe operator $L$ (resp. of the Dirac operator squared $\Dirac^2$) is maximized
(respectively minimized) exactly when the metric $g$ is the pullback
of $g_{\rm{can}}$ by some conformal diffeomorphism on $(S^6, g_{\rm{can}})$.
\end{theorem}

For readers not familiar with the intricacies of explicit Polyakov and $Q$-curvature formulae, it is worth pointing out that the local Riemannian invariants that need to be handled increase tremendously in combinatorial complexity, and that the analogue of Theorem \ref{BransonsTheorem} has not been proven for $n\geq 8$ (see \cite{GP} for relevant formulae for the case $n=8$).

In the light of this complicated nature, it is striking that one may obtain certain types of quite general information about the extremal properties of the determinant, even in high dimensions. Namely, a novel line of investigation (that dealt with second order geometric Laplace-type operators) was recently initiated by K. Okikiolu \cite{Ok3}
(and K. Richardson
\cite{Ri}). The extremals are now local in the metric, near a fixed ground metric, but with respect to any
variation, i.e. not restricting to only conformal variations of the
metric. The relevant object is the stability operator, i.e. the $L^2$-Hessian, of $F_P$, at the ground metric $g$.
\beq
\Hess{F_{P_g}(k,k)}:=D^2F_{P_g}=\frac{d^2}{dt^2}_{\big\vert t=0}F_{P_{g+tk}},\quad k\in\CinfStwoM.
\eeq
where $S^2M$ denotes the bundle of symmetric covariant 2-tensors. In the paper \cite{Ok3} from Ann. Math. (2001), and further elaborated in \cite{Ok4}), K. Okikiolu developed a calculus showing that such a stability operator (i.e. $L^2$-Hessian) may in many cases be understood properly in terms of a corresponding classical polyhomogeneous pseudodifferential operator $Q$ as follows.
\begin{equation}
\Hess{F_{P_g}(k,k)}=\DIP{Qk}{k}_g,
\end{equation}
where the inner product induced on sections by $(g,\langle\cdot,\cdot\rangle)$ is
\beq\label{DIPdef}
\DIP{\varphi}{\psi}_g=\int_M\langle\varphi,\psi\rangle_x \rm{dvol}(x),\quad\textrm{for}\quad \varphi,\psi\in\Cinf{E},
\eeq
denoting by $\rm{dvol}$ the Riemannian measure of $(M,g)$.

This analysis (in \cite{Ok3}) showed that for the determinant of the Yamabe operator (i.e. conformal Laplacian) at $(S^{2k+1}, g_{\rm{can}})$, by positivity of such a stability operator, the leading part of which is a locally determined object (i.e. in jets of the metric and the symbol of the partial differential operator), there holds local extremality near the round metric (in an appropriate topology such as a Banach topology on fixed Sobolev spaces, or the Fréchet topology). The maxima and minima are strict apart from in certain natural gauge-invariance directions. The Laplace-Beltrami however may have saddle points (though only finite-dimensional). These and other examples show that the conformal properties of the operators play an important rôle.

The following table summarizes all previously known results on conformally covariant operators in the case of the round sphere, together with the results of the present paper. Note that the $2k$-dimensional cases, for $k\geq2$, are only known to be true up to a finite codimension of exceptional directions, after taking the quotient with the gauge invariant directions as in Equation (\ref{DetModuli}). \\

\rule[8pt]{0pt}{0pt}
\begin{tabular}{l @{\quad} l @{\quad} l @{\quad} l}
\toprule
$S^n$ &
$\det L$ &
$\det\Dirac^2$ &
Fixed quantities \\
\midrule
$S^2$ & global max. & global min. & volume\\
$S^4$ & global min. & global max. & volume + conformal class\\
$S^6$ & global max. & global min. & volume + conformal class\\
$S^{4k}\:\quad(k\geq 2)$ & local min.$(\dagger)$ & local max.(*) & volume\\
$S^{4k+1}$ & local min. & local max.(*) & volume\\
$S^{4k+2}\: (k\geq 2)$ & local max.$(\dagger)$ & local min.(*) & volume\\
$S^{4k+3}$ & local max. & local min.(*) & volume\\
\hline
\multicolumn{2}{
>{\raggedright\footnotesize}p{5cm}}{
 $(\dagger)$: \cite{Ok5} + the present paper.\\$(*)$: the present paper.
}
\end{tabular}

The pattern seen in this table for $S^2$, $S^4$ and $S^6$ lead T. P. Branson to conjecture the following (where it should be noted that the
difference in behavior between the Yamabe and Dirac operator is not merely due to unnatural sign conventions).

\begin{conjecture}[Branson's conjecture, 1993, \cite{BrFuncDet}]
On $S^n$ for $n$ even, the pattern continues, i.e. in the
conformal class of the standard sphere, the quantities
$(-1)^{n/2}\det L$ and \mbox{$-(-1)^{n/2}\det\Dirac^2$} are minimized at the standard metric. The extremal metric is unique up to conformal diffeomorphism pullback.
\end{conjecture}

With the previous state of knowledge about the extremal properties of the determinant, it thus seemed remarkably fitting to investigate the determinant of the Atiyah-Singer-Dirac operator. As hinted at in the introduction, Dirac operators are certainly important objects in both modern theoretical physics and pure geometry. The present paper serves several purposes: (1) to extend the stability operator (i.e. $L^2$-Hessian) calculus for the extremal problems of spectral invariants to the geometrically much more involved case of Dirac operators, and the independent issue: (2) to give a rigorous spectral theory for stability operators in variational problems suitable for treating zeta regularized quantities in even dimensions, which involves a more broad class of pseudodifferential operators, namely having a log-polyhomogeneous leading symbol. This is also a new case already for the conformal Laplacian, when the dimension of the manifold is even, and the spectral results proven here concludes the proof of $(\dagger)$ in the above table, using the formulas for the leading symbols of the ordinary Laplacian in \cite{Ok3} and \cite{OW}, together with variation formulae for the scalar curvature.

Furthermore we describe as an aside: (3) an approach to breaking gauge invariance of the variational problem, in order to apply elliptic theory, by noting that each of the relevant operators has a factorization (up to a smoothing operator) into a product of one truly elliptic operator and two projections that project onto the gauge invariance subspace of the smooth symmetric 2-tensor fields (as defined by the usual differential conditions).

Taken together this leads to the affirmative answer to Branson's conjecture, with a certain amount of liberty in its interpretation; viz. by proving the new extremal results marked $(\dagger)$ and $(*)$ in the above table (see Theorem \ref{DiracExtremals}, where again this is in the sense of extremals in ``almost every direction'', or more precisely that there is possibly a finite co-dimension of directions, up to gauge equivalence, in which the extremality does not hold.

Finally, the present paper is part of a larger program of understanding extremal properties of determinants in general, and it is an important prerequisite for the later joint work with B. \O{}rsted \cite{MoO}, where the complete local variations characterization of the extremals for the Dirac operator on the round spheres $(S^n,g_\textrm{can})$ becomes a corollary to the present paper, when applying to it the strong rigidity theorem for conformal functionals on $(S^n,g_\textrm{can})$ proven in \cite{MoO}.

\subsection{Proof summary and statement of the results}\label{SummaryAndResults}
Let $(M,g)$ be a closed (i.e. compact, $\partial
M=\emptyset$) smooth Riemannian $n$-dimensional manifold and $(E,\langle\cdot,\cdot\rangle)$ a
Hermitian vector bundle over $M$ of rank $r$. We shall assume that $E$ is a
\emph{tensor-spinor bundle}, i.e. if $H$ is either $\Oof{n}$, $\SO{n}$
or $\Spin{n}$, corresponding to the geometry under investigation, then
\[
E=\mathcal{F}_HM\times_\rho V
\]
is a bundle associated to the principal bundle of $H$-frames by some
finite-dimensional representation $\rho$ on $V$.

Consider partial differential operators
\[
P:\Cinf{E}\to\Cinf{E}
\]
of order $d\in 2\Naturals_{+}$, acting on smooth sections and satisfying the following.
\begin{analassump}\label{analassump}
\begin{itemize}
\item[]
\item[(1)] P has pointwise positive definite leading symbol $\sigma_d(x,\xi)\in\Endo{E_x}$
\item[(2)] P is formally positive (in particular formally self-adjoint) on
  $\Cinf{E}$, i.e.
\[
\DIP{P\varphi}{\psi}_g\geq 0, \quad\forall \varphi,\psi\in\Cinf{E}
\]
\end{itemize}
\end{analassump}
\noindent{}It is a classical fact that both the ordinary Laplacian and
the square of the Dirac operator satisfy this. Under the assumptions, $P$ is elliptic and has discrete
spectrum with eigenvalues
\[
0\leq\lambda_0\leq\lambda_1\leq\cdots\nearrow\infty
\]
of finite multiplicity and satisfying Weyl's law, which ensures
convergence in the following definition.
\begin{definition}
The zeta function of $P$ in the metric $g$ is
\begin{equation}
\zeta_P(s)=\zeta(P,s)=\sum_{\lambda_j>0}\lambda_j^{-s},\quad\Real\; s>n/d,
\end{equation}
with repetition according to multiplicities.
\end{definition}
Standard theory (see e.g. \cite{BO1}) gives the meromorphic structure of $\zeta$ as
recorded in the next theorem.
\begin{theorem}\label{mero}
$\zeta_P(\cdot)$ has a meromorphic continuation to $\Complex$ having only
simple poles with
\begin{equation}
\Big\{\mathrm{poles\;of}\;\zeta_P(\cdot)\Big\}\subseteq\bigg\{\frac{n}{d},\frac{n-1}{d},\ldots\bigg\}
\end{equation}
and being regular at $s=0$. If $n$ is odd-dimensional, then
\begin{equation}
\zeta_P(0)=-\dim{\ker{P}}.
\end{equation}
\end{theorem}
\noindent{}In the light of this, it is convenient to define the modified zeta function
\begin{equation}\label{modifiedzeta}
\mathcal{Z}(P,s)=\frac{\Gamma(s)\zeta_P(s)}{\Gamma(s-n/2)}+\frac{\dim{\ker{P}}}{s\Gamma(s-n/2)},
\end{equation}
which by Theorem \ref{mero} is entire in $s$.

Also by the same theorem the following makes sense
\begin{definition} The zeta functional determinant of $P_g$ is the real number
\[
\det{P_g}=\exp\big(-\zeta_P'(0)\big).
\]
\end{definition}
We will also require operators to be geometric, for instance in the
sense of Branson-\O{}rsted and others (see for instance \cite{BO1}).
\begin{naturassump}\label{naturassump}
$P$ is assumed to be natural as 'a rule' assigning to each
metric the operator $P_g$, being a universal polynomial in tensor products and contractions of
\begin{itemize}
\item[(1)] the metric $g$, its inverse $g^{-1}$, the Levi-Civita connection
  $\nabla$ and the curvature tensor $R$.
\item[(2)] the volume form $\vol$, if the structure group is $H=\SO{n}$.
\item[(3)] the volume form and Clifford section $\gamma$, if $H=\Spin{n}$.
\end{itemize}
\end{naturassump}
\noindent{}Differential operators $P$ will in the following be assumed to satisfy the
analytic and naturality conditions, unless otherwise stated.

Using K. Okikiolu's methods (\cite{Ok4}, \cite{Ok3}), we study the stability operator (i.e. $L^2$-Hessian) of the modified zeta function $\mathcal{Z}(s)$ which has a meromorphic family of Hessians $\Psi$DOs denoted $T_s$, for $s\in\Omega\subseteq\Complex$, with decompositions $T_s=U_s+V_s$ into pairs of $\Psi$DOs with orders $n-2s$ and $2$, respectively. The main result following from this analysis is the following.

\begin{maintheorem}\label{DiracHess}
Let $(M^n,\gamma)$ be a closed Riemannian spin manifold. Assume that the kernel of the Atiyah-Singer-Dirac operator $\Dirac$ has stable
dimension under local variations of the metric, with fixed topological
spin structure.
\newline\indent{}Then the leading symbol of the part $U_s$ as above, of the pseudodifferential stability operator of the modified zeta function $\mathcal{Z}(s)$ of $\Dirac^2$ is given by
\begin{align*}
&\langle
k,u_s(x,\xi)\,k\rangle_g=\\
&2^{\lfloor\frac{n}{2}\rfloor-2}\Big(\frac{1}{4\pi}\Big)^{\fracsm{n}{2}}\frac{\Gamma(-S+1)^2}{\Gamma(-2S+2)}|\xi|^{n-2s}
\Bigg\{\Big[2s-(n-1)\Big]\tr\big(K_g\Pi^{\bot}_\xi\big)^2
+\big(\tr{K_g\Pi^{\bot}_\xi\big)^2}\Bigg\}
\end{align*}
for $\Real{s}<n/2-1$. Here $S=s-n/2$, and $\Pi^{\bot}_\xi$ is the orthogonal projection on ${\xi}^{\perp}$, for $\xi\in T_x^{*}M$.
\end{maintheorem}
\begin{remarks}
We point the reader to the paper \cite{Mol2} for an application, which also serves as a simple-minded consistency check of the correctness of Theorem \ref{DiracHess}.
\end{remarks}

It is worth pointing out that the previous investigation of higher rank vector bundle cases, namely Hodge and Bochner Laplacians on $p$-forms in \cite{OW} led to somewhat complicated expressions for the stability operator. These operators lacked good conformal properties, and working instead with the square of the conformally covariant Atiyah-Singer-Dirac operator indeed reveals a very appealing formula.

As prerequisites for proving Theorem \ref{DiracHess}, we describe in Section \ref{Section:Diracoperators} the basics of spin geometry and the non-trivial problem of metric variations of the Dirac operator in a general direction, i.e. not necessarily preserving the conformal class. The situation is more complicated than that for the ordinary Laplacians that act on sections of fixed, metric independent bundles,
such as functions or differential forms. Interestingly there seems to be the misconception that only conformal changes are at all manageable for the spin case. This should seemingly be attributed to the fact that a \emph{fixed} principal bundle may be used in the conformal case only, and to the fact that the spin representations do not come from representations of the universal double cover of $\Gl{n}$, but only from that of $\SO{n}$, corresponding to \emph{after} a Riemannian metric has been chosen. It is indeed true that for a general change of metric, the underlying bundles have to change. However being in Riemannian signature, and for the purpose of spectral geometry, this can be handled (following \cite{BG}) via families of gauge transformations known as the Bourguignon-Gauduchon isomorphisms. One thus obtains a new family of partial differential operators
\begin{equation}\label{isofamily}
\Diracpresuffix{\gamma}{\gamma_t}:C^\infty(\Sigma_\gamma)\to C^\infty(\Sigma_\gamma),\quad t\in (-\varepsilon,\varepsilon),
\end{equation}
in the \emph{fixed} spinor bundle (for the metric $g$, together with the spin structure denoted a ``spin metric'' $\gamma$), and each of which is isospectral to the Dirac operator in the metric $g+tk$, for $k\in S^2M$. For readers that might not be familiar with these subjects, is included a review of the Bourguignon-Gauduchon paper \cite{BG} in Appendix \ref{BGAppendix}.

The proof of Theorem \ref{DiracHess} is carried out in Section \ref{hessiansDirac} by generalizing K. Okikiolu's calculus to the spinor case (Corollary \ref{OkikioluForDirac}), and by calculating the explicit leading symbol of the stability operator (i.e. $L^2$-Hessian) in local coordinates, by applying the Bourguignon-Gauduchon formula from \cite{BG} for the infinitesimal variation of the operator family in (\ref{isofamily})
\beq\label{BGVarIntro}
\bigg({\frac{d}{dt}\Diracpresuffix{\gamma}{\gamma_t}}_{\big\vert
    t=0}\bigg)\psi=-\frac{1}{2}\sum_{i=1}^{n}e_{i\cdot\gamma}\widetilde{\nabla}_{K_{g}(e_i)}^{\gamma}\psi+\frac{1}{4}\big[d(\tr_g k)-\diverg_g
    k\big]_{\cdot\gamma}\psi.
\eeq
The final step in the proof of Theorem \ref{DiracHess} is the use of certain formulae for the trace of endomorphisms induced from Clifford multiplication (Proposition \ref{CliffTrace}).

As an outline of the arguments needed for the proof of the main result on extremals for the determinant (i.e. Theorem \ref{DiracExtremals}), we now explain the application of Theorem \ref{DiracHess} to obtain the extremals in a simpler case, being that of the value $\zeta_{\Dirac^2}(0)$.

\begin{corollary}\label{ZetaItself} 
Assume that the ground metric $g_0$ is a stationary point for $\zeta_{\Dirac^2}(0)$. Under assumptions as in Theorem \ref{DiracExtremals}, then $\zeta_{\Dirac^2}(0)$ has local maximum for $(-1)^k\zeta_{\Dirac^2}(0)$ at $g_0$, apart from possibly in $V(M,g_0)+\confdiff$, for a finite dimensional subspace of directions $V(M,g_0\subseteq\CinfStwoM$.
\end{corollary}

\begin{remark}\label{DefExtremal}
By extremality at $g_0$ of a functional $F$ on the space of metrics, apart from in the directions $V(M,g_0)+\diff$, we mean as follows: If $g_t$ is a $C^\infty$-curve of Riemannian metrics with
\[
k:=\frac{d}{dt}_{\big|t=0}g_t \in (V+\diff)^\perp=\diff^\perp\cap V^\perp,
\]
where $\perp$ designates the orthogonal $L^2$-complement as in (\ref{DIPdef}). Then there exists $\delta=\delta(F,k)>0$ such that
\beq\label{WeakEx}
0<|t|<\delta\Rightarrow F(g_t)>F(g_0).
\eeq
\end{remark}

\begin{proof}[Proof of Corollary \ref{ZetaItself}]

To apply ellipticity arguments in this situation, we need to break the gauge invariance. This can be done by factorizing out the projections onto the invariant directions, as described in detail in Section \ref{projections}. Namely, for this we solve a system of geometric elliptic equations (Proposition \ref{projectionsymbols}) to find the explicit orthogonal projection onto a certain subspace $\diffperp$ of the tangent space of the Riemannian manifold of Riemannian metrics
\[
\projdiffperp: \CinfStwoM\to\diffperp
\]
as a $\Psi$DO of order 0, defined up to smoothing operators, where $\diff=\diffspin$ is the tangent space of the pullback of metrics by diffeomorphisms
\beq
\diffspin=\Big\{\frac{d}{dt}_{\big\vert t=0}\varphi^*_tg_0\:\Big\vert\:\varphi_t\:\textrm{is a 1-param. family of spin-diffeomorphisms}\Big\},
\eeq
Note that the space $\diffspin$ consists by (\ref{DetModuli}) of zero directions for the stability operator. The result is
\begin{align*}
&\projdiffperp=\Id-\nabla^\odot\Big[\diverg\nabla^\odot\Big]^{-1}\diverg,\\
&\sigma_L\Big(\projdiffperp\Big)(x,\xi)K=\Pi_\xi^\perp K\Pi_\xi^\perp.
\end{align*}
where $\nabla^\odot$ is the symmetrized covariant derivative. The projection $\projconfdiffperp$, relevant when there is conformal invariance (e.g. functional determinant in odd dimensions), is treated similarly in Proposition \ref{projectionsymbols}, where
\beq
\conf=\Big\{\varphi g_0\:\Big\vert\:\varphi\in\Cinf{M}\Big\}\quad\textrm{and}\quad\conf^\perp=\Big\{k\in\CinfStwoM\:\Big\vert\:\tr_{g_0}k=0\Big\},
\eeq
and the leading symbol of the projection is
\beq\label{ProjConfLeading}
\sigma_L\Big(\projconfdiffperp\Big)(x,\xi)K=\Pi_\xi^\perp K\Pi_\xi^\perp-\frac{1}{n-1}\tr\big(\Pi_\xi^\perp K\big)\Pi_\xi^\perp
\eeq

From Theorem \ref{DiracHess} and the formula (\ref{ProjConfLeading}), the leading symbol of $\zeta_{\Dirac^2}(0)$ is
\[
u_0(x,\xi)K=(n-1)2^{\lfloor\frac{n}{2}\rfloor-2}\Big(\frac{1}{4\pi}\Big)^{\fracsm{n}{2}}\frac{\Gamma(\frac{n}{2}+1)^2}{\Gamma(n+2)}|\xi|^n\projconfdiffperp.
\]

Then using the factorization result (Proposition \ref{factorizing}), since $\zeta_{\Dirac^2}(0)$ is conformally invariant for $n$ even, we can write for the stability operator (i.e. $L^2$-Hessian)
\beq
\Hess\zeta(0)=\projconfdiffperp H_0\projconfdiffperp
\eeq
in even dimension $n=2k$, where $H_0$ is a new pseudodifferential operator, now with a chance of being an elliptic operator. In the applications in this paper, the new operator $H_0$ is in fact elliptic, and $H:=(-1)^{k}H_0$ furthermore has negative definite leading symbol (in other situations however, this may not be the case, as the examples of the Laplacian on forms in \cite{OW}, and the half-torsion in \cite{Mol2} show), and is symmetric with respect to the $L^2$-inner product.

The results anticipated in the above discussion suggest taking $V$ to be the finite-dimensional subspace
\beq\label{Vdef}
V:=\bigoplus_{\lambda_k\geq 0}E_k(H)
\eeq
of non-negative eigenspaces for the elliptic pseudodifferential operator with positive leading symbol $H$. Namely, by the spectral theory for such operators (which is a standard consequence of the compactness of the resolvent, e.g. also a special case of Theorem \ref{spectrum} in the present paper) on the closed manifold $M$, one finds that $H$ has finite multiplicity, discrete, pure eigenvalue spectrum with $|\lambda_k|\to\infty$, and that the spectrum is bounded from above,
\beq
\spec{H}\subseteq (-\infty,c],\quad c>0,
\eeq
and consequently there is at most a finite number of non-negative eigenvalues, each with a corresponding finite-dimensional eigenspace, and hence
\[
\dim V <\infty.
\]

To verify our definition of $V$, note that for $k\in(V+\confdiff)^\perp\backslash \{0\}$,
\beq\label{Hneg}
\Hess F(k,k)=\DIP{\projconfdiffperp k}{H\projconfdiffperp k}=\DIP{k}{Hk}<0,
\eeq
using that $k\in\confdiff^\perp$ and $k\in V^\perp$, respectively.

Thus, since $g_t$ is assumed to be a smooth curve of metrics, the local extremality claim follows from Taylor's formula with remainder for real smooth functions (and of course it is even a strict local extremum).
\end{proof}

For the more complicated case of the determinant of $\Dirac^2$ and of $L$, or equivalently of $\zeta'_{P}(0)$, one would like to apply an approach similar to the one in the proof of Corollary \ref{ZetaItself}, but must here take the $s$-derivative at $s=0$ in the expression in Theorem \ref{DiracHess}. To treat this rigorously it is convenient to work in the log-polyhomogeneous symbol class, recently introduced by M. Lesch in \cite{Le}. We review this in Section \ref{logsection} together with the needed theory of holomorphic families of pseudodifferential operators. This is a relatively new class of operators, and we present it with slightly weaker assumptions than previous authors (\cite{Le}, \cite{KV}, \cite{PS}) by the use of Cauchy estimates in the holomorhpic parameter.

Namely, in even dimensions the rôle of the log-polyhomogeneous class is essential, since the leading symbol of the stability operator of $\det\Dirac^2$ is of the form
\beq\label{logform}
|\xi|^nA(x,\xi)\log|\xi|+|\xi|^nB(x,\xi),\quad A, B\in C^\infty(T^*M, \Endo{TM}).
\eeq
The difference between odd and even dimensions originates in the $\Gamma(s-n/2)$ factor in (\ref{modifiedzeta}), which is analytic in $s$ at $s=0$ if $n$ is odd, while having whenever $n$ is even a simple pole at $s=0$ (see Lemma \ref{gammalemma}).

To deduce the extremal results for the determinant, we need spectral results similar to those discussed in the proof of Corollary \ref{ZetaItself} above. In Section \ref{logsection} we define the appropriate notions of hypoellipticity and positivity for log-polyhomogeneous symbols, while Theorem \ref{spectrum} gives the main spectral result, namely the finite index propery in the class of symmetric log-polyhomogeneous operators with positive leading symbol. The main ingredient in proving this is a G\aa{}rding inequality (Corollary \ref{Gaarding}), for hypoelliptic symbols, and the simple observation that operators with positive leading symbol have square roots in the hypoelliptic class of half the bi-degrees (Lemma \ref{squareroot}).

In Section \ref{projections} we establish the explicit factorization results mentioned above (Proposition \ref{factorizing}). Finally we show that the explicit symbols of the form in (\ref{logform}), coming from Theorem \ref{DiracHess}, are in the new class. This happens to be a non-trivial point, because for log-polyhomogeneous symbols the correct notion of leading symbol includes both terms in Equation (\ref{logform}), and not only the highest log-degree. In fact in this application, the endomorphism $A$ will be singular, but since $B$ is regular and sufficiently large, we can prove hypoellipticity (Proposition \ref{hypoandpositive}).

This will conclude the proof of the second main theorem of the present paper, where the extremals are again to be interpreted as in Remark \ref{DefExtremal}. Recall that the determinant is always invariant in certain large subspaces of directions, for even and odd dimensions being $\diffspin$ and $\confdiffspin$, respectively.

\begin{maintheorem}\label{DiracExtremals}
Let $(M^n,\gamma)$ be a closed Riemannian spin manifold with $n\geq 3$. Consider local, volume preserving variations of the metric with fixed topological
spin structure. Assume that the kernel of its
Atiyah-Singer-Dirac operator $\Dirac$ has stable
dimension under these variations, and that the ground metric $g$ is a stationary point of $\det\Dirac^2$.
\newline\indent{}
Then, apart from in $V(M,g_0)+\diff$ when the dimension is even (respectively $V(M,g_0)+\confdiff$ when the dimension is odd), for a finite dimensional (possibly empty) subspace of directions $V\subseteq\CinfStwoM$, the metric $g$ is a local maximum for $(-1)^{\lfloor n/2\rfloor}\det\Dirac^2$.
\end{maintheorem}

\begin{remarks}
\begin{itemize}
\item[] 
\item[(1)] The restriction $n\geq 3$ is technical and is due to the
fact that we need $U_s$, which is the locally computable part of the symbol, to be
the leading symbol near $s=0$.
\item[(2)] We point the reader to other results concerning such alternating $\modulo 4$ patterns for zeta regularized quantities, in the paper \cite{Mol1}. E.g. the sign of $\log\det(\Dirac^2,S^n)$ is $(-1)^{\lfloor(n-1)/2\rfloor}$, and $\lim_{n\to\infty}\det(\Dirac^2,S^n)=1$.
\item[(3)]
The reason for the ``opposite'' behavior for the determinants of the Dirac and conformal Laplacians is not completely evident, as has been discussed previously from a geometric viewpoint in \cite{BrFuncDet}. See \cite{Mol2} for a connection with theoretical physics, which is at least consistent with this behavior.
\end{itemize}
\end{remarks}

Specializing again to spheres $S^n$, which are spin manifolds, we note that the standard, round metrics are stationary points for $\det\Dirac^2$ (see e.g. Proposition \ref{stationarySpheres} in the present paper for a proof that works in any dimension, or see \cite{MoO} and \cite{Bl} for a simpler proof in the conformally invariant situations). Theorem \ref{DiracExtremals} applies in this case, namely the dimension of the kernel is constant, since the round sphere, and hence metrics close to it in the Fréchet space topology have no harmonic spinors (though it is known that sufficiently far away from the round metric on $S^n$ such spinors do exist, e.g. \cite{Da}). Thus we obtain as promised a partial verification of this version of Branson's Conjecture.

One important application of the results proven here is to the round spheres $(S^n,g_0)$, which by Proposition \ref{stationarySpheres} are stationary points of $\zeta_{\Dirac^2}(s)$, for any $s\in\Complex$ a regular point of $\zeta_{\Dirac^2}$, and hence the results apply. As mentioned in the introduction, in later joint work with B. \O{}rsted (\cite{MoO}) we utilize Theorem \ref{DiracHess} as a key component for obtaining stronger statements, namely without the exceptional directions. In particular $V(S^{2k+1},g_0)=\{0\}$ for the determinant. The proof relies on a remarkable rigidity theorem for conformal functionals, which in turn is proved using the semisimple Lie theory of the conformal group $\SO{n+1,1}$. The results obtained are:

\begin{maintheorem}[\cite{MoO}]
Among metrics on $S^{2k}$ of fixed volume, the standard sphere $(S^{2k},g_{0})$ is a local maximum for $(-1)^k\zeta_{\Dirac^2}(0)$.
\end{maintheorem}

\begin{maintheorem}[\cite{MoO}]
Among metrics on $S^{2k+1}$ of fixed volume, the standard sphere $(S^{2k+1},g_{can})$ is a local maximum for $(-1)^k\det\Dirac^2$. Furthermore, apart from in the natural invariant directions $\conf_{g_0}+\diff_{g_0}$, the maximum is strict.
\end{maintheorem}

The proof given in the follow-up paper \cite{MoO} is quite in the spirit of T. P. Branson's work on the determinant, and spells out clearly (comparing to the analogous paper for the determinant of the conformal Laplacian in \cite{Ok1}, which also becomes a special case of \cite{MoO}), why such results are true. Namely in odd dimensions determinants of integer powers of conformally covariant operators (with trivial kernels), are conformally invariant functionals. Furthermore the conformal group of the standard $S^n$ is large (it has maximal dimension), in the sense that it acts irreducibly on the relevant quotient. The argument given there is thus general and shows that on the round sphere the situation is very rigid, and one can in great generality expect that for conformally invariant functionals, the exceptional subspaces for determinants are trivial on the sphere,
\[
V(S^n,g_0))=\{0\}.
\]

\section{Dirac operators and non-conformal change of metrics}\label{Section:Diracoperators}
Let $M$ be an oriented Riemannian manifold of dimension $n$. Let
\[
\mathcal{F}_{\mathrm{SO}}M=\mathcal{F}_{\mathrm{SO}}(TM)
\]
be the bundle of oriented orthogonal frames. This is in a natural way a principal $\SO{n}$-bundle. Recall the following bundle theoretical notion of spin structure.
\begin{definition}
A spin structure on $M$ is a principal
$\Spin{n}$-bundle $\PSpin M$ that $2$-fold covers $\mathcal{F}_\mathrm{SO}M$ as
$G$-bundles, that is:
\begin{equation*}
\xymatrix @R=0.5pc{
\PSpin M\times\Spin{n}\ar[r]\ar[dd]_{\Pi\times\pi} & \PSpin M\ar[dd]_{\Pi}\ar[dr]^{\pi_\mathcal{P}} & \\
  &   & M\\
\mathcal{F}_\mathrm{SO}M\times\SO{n}\ar[r] & \mathcal{F}_\mathrm{SO}M\ar[ur]_{\pi_{\mathcal{F}}} &}
\end{equation*}
If $M$ admits a spin structure it is said to be a spin manifold.
\end{definition}
Though defined in a Riemannian geometric setting, being a spin manifold is a differential topological property, and the
spin structure is independent of the metric $g$ up to a certain equivalence of
spin structures (i.e. as principal bundle coverings). Here the interest is in changing the underlying Riemannian
metrics. Recall that whether $M$ is spin or not can be read off from the second
Stiefel-Whitney class in the second $\mathrm{\check{C}}$ech
$\mathbb{Z}_2$-cohomology of $M$, $w_2(TM)$, which vanishes if and
only if $M$ is spin. When this is the case, then the number of inequivalent spin
structures is related to the fundamental group (assuming $M$ is connected) by
\[
\Big|\Big\{\mathrm{inequivalent\; spin\; structures}\Big\}\Big|=\Big|\mathrm{Hom}(\pi_1(M),\mathbb{Z}_2)\Big|
\]

Recall that the spheres $S^n$ are examples of spin manifolds, and if $n\geq 2$ the spin structure is unique (up to equivalence).

\begin{definition}
If $M$ is spin then the spinor bundle is the complex vector bundle
\[
\Sigma M=\mathcal{P}_\mathrm{Spin} M\times_\rho\Complex^{2^k},
\]
where the associating $\rho$ is as the spinor representation, and $n=2k$ or $n=2k+1$.
\end{definition}
The bundle $\Sigma M$ naturally has a Hermitian structure, denoted in each fiber simply by $\langle\cdot,\cdot\rangle_x$. This comes about by
the usual procedure through averaging the standard Hermitian inner product
on $\Complex^{2^k}$ with respect to a certain finite group that generates $\CliffC_n$, and so this invariant inner product descends to the associated bundle $\Sigma M$.

Recall also the definition of the associated Dirac operator.
\begin{definition}\label{DiracDef}
The Atiyah-Singer-Dirac operator $\Dirac$ is the composite map
\[
\xymatrix {\Cinf{\Sigma M}\ar[r]^{\hspace{-1.7em}\widetilde{\nabla}}&
\Cinf{T^*M\otimes\Sigma M}\ar[r]^{\hspace{0.2em}\#}&\Cinf{TM\otimes\Sigma M}\ar[r]^{\hspace{1.2em}\cdot_\gamma}&\Cinf{\Sigma M}},
\]
where $\cdot_\gamma$ denotes Clifford multiplication.
\end{definition}

\begin{theorem}[\cite{BG}]
Let $\gamma$ and $\eta$ be two spin metrics corresponding to the same
topological spin structure (i.e. to two corresponding reductions of the structure group from $\widetilde{\mathrm{G}}\mathrm{l}^+(n)$ to $\SO{n}$, see the Appendix).

There is a bundle map $\beta$ between spinor bundles
\[
\beta_\eta^\gamma:\Sigma M_\gamma\to\Sigma M_\eta,
\]
which is equivariant so that it induces a map on smooth sections (i.e. spinor fields), still denoted in the same way:
\[
\beta_\eta^\gamma:\Cinf{\Sigma M_\gamma}\to\Cinf{\Sigma M_\eta},
\]
with the properties:
\begin{itemize}
\item[(1)] $\beta$ is a fiberwise isometry of Hermitian vector bundles.
\item[(2)] $b$ and $\beta$ are
compatible with Clifford multiplication, in the sense that
\beq\label{betaCliff}
\beta_\eta^\gamma(X\cdot_\gamma\varphi)=b_h^g(X)\cdot_\eta\beta_\eta^{\gamma}(\varphi),
\eeq
where $b$ is the natural map .
\end{itemize}

\end{theorem}

The gauge transformed Dirac operator may now be described. Fixing a
topological spin structure and spin metrics $\gamma$ and $\eta$
corresponding to this and the metrics $g$ and $h$ respectively, we let
\begin{equation}
\Diracpresuffix{\gamma}{\eta}=\big(\beta_\eta^\gamma)^{-1}\circ\Dirac^\eta\circ \beta_\eta^\gamma.
\end{equation}
Note that this operator
\[
\Diracpresuffix{\gamma}{\eta}:\Cinf{\Sigma M_\gamma}\to\Cinf{\Sigma M_\gamma}
\]
is manifestly not the Dirac operator corresponding to the spinor metric $\gamma$. Rather it is an operator acting canonically on
$\gamma$-spinors but having the same eigenvalues as the Dirac operator
in the spin-metric $\eta$. The infinitesimal variation is given by the following theorem (see the Appendix for a review of the proof).

\begin{theorem}[\cite{BG}]\label{varDirac}
The infinitesimal variation in the metric direction $k$ of the Dirac operator for spinorial
metric $\gamma$ is
\begin{equation}
\bigg({\frac{d}{dt}\Diracpresuffix{\gamma}{\gamma_t}}_{\big\vert
    t=0}\bigg)\psi=-\frac{1}{2}\sum_{i=1}^{n}e_{i\cdot\gamma}\widetilde{\nabla}_{K_{g}(e_i)}^{\gamma}\psi+\frac{1}{4}\big[d(\tr_g k)-\diverg_g
    k\big]_{\cdot\gamma}\psi,
\end{equation}
where $(e_i)$ is a $g$-orthonormal frame.
\end{theorem}
\begin{remark}
The divergence $\diverg_g$ of the 2-form $k$ with respect to $g$, is the covariant derivative followed by contraction, i.e. with no minus sign.
\end{remark}

\subsection{Proof of stationarity of $\det\Dirac^2$ on round spheres}
In the following proposition we observe that all round spheres are stationary points for the determinant of the Dirac squared under general volume preserving variations.

\begin{proposition}\label{stationarySpheres}
The round spheres $(S^n,g_{\textrm{can}})$ are stationary points (i.e. critical points) of
the functional $\det\Dirac^2$, with respect to volume preserving
metric variations. In fact the whole zeta function is pointwise stationary.
\end{proposition}
\begin{proof}
The spheres $S^n$ ($n\geq 2$) are simply connected spin
manifolds, and thus each has a unique topological spin structure. From a result by Bourguignon-Gauduchon (Proposition 29 in \cite{BG}) on standard round $S^n$, it is known that if we look at a specific eigenvalue $\lambda$ of the Dirac
operator in the standard metric, then the standard metric is stationary for the sum of those eigenvalues $\lambda_{(1)},\ldots,\lambda_{(n_\lambda)}$ with
multiplicities, that branch from $\lambda$ under
perturbation.

Thus, working in the halfplane $\Real s>n/2$, we apply the absolutely convergent sum
representation of the zeta function. Perturbing along a smooth curve of Riemannian metrics $g_t$
we get
\[
\frac{\partial}{\partial t}_{\big\vert
  t=0}\Big(\zeta(\Dirac^2_{g_t},s)\Big)=
\frac{\partial}{\partial t}_{\big\vert
  t=0}\bigg(\sum_\lambda\sum_{j=1}^{n_\lambda}\lambda_{(j)}^{-s}(t)\bigg)=
-s\sum_\lambda\bigg\{\lambda^{-s-1}\sum_{j=1}^{n_\lambda}\lambda_{(j)}'(0)\bigg\}=0,
\]
when $\Real s>n/2$. The analytical continuation of this to $\Complex$ is identically zero, which proves the claim.
\end{proof}

\section{Gauge-breaking projections as pseudodifferential operators}\label{Spindiff}
In considering stability operators for zeta functions of operators of ordinary Laplacian type, the natural subspace of invariant tangent directions in the manifold of Riemannian metrics is
\beq
\diff=\Big\{\frac{d}{dt}_{\big\vert t=0}\varphi^*_tg_0\:\Big\vert\:\varphi_t\:\textrm{is a 1-param. family of diffeomorphisms}\Big\}
\eeq
In the context of Dirac operators the natural subspace is
\beq
\diffspin=\Big\{\frac{d}{dt}_{\big\vert t=0}\varphi^*_tg_0\:\Big\vert\:\varphi_t\:\textrm{is a 1-param. family of spin-diffeomorphisms}\Big\}
\eeq
For local variations it suffices to consider a fixed topological spin structure, since naturally we have the following.
\begin{proposition}
\[
\diff=\diffspin.
\]
\end{proposition}
\begin{proof}
We may, by composition with $\varphi_0^{-1}$, assume that $\varphi_0=\Id_M$, which is spin-preserving. Since spin structures form a discrete topological space, every $\phi_t$ in the smooth curve from the identity in $\Diff(M)$ is spin-preserving, and the two spaces coincide.
\end{proof}

\indent{}
The subspaces $\diffperp$ and $\confdiffperp$ are closed inside $L^2$-sections, since they are defined by elliptic differential conditions, and we shall need the following proposition giving the orthogonal projections, and thus their explicit leading symbols, explicitly.
\begin{proposition}\label{projectionsymbols}
The projection maps
\begin{align*}
\projdiffperp&: \CinfStwoM\to\diffperp,\\
\projconfdiffperp&: \CinfStwoM\to\confdiffperp
\end{align*}
are 0th order classical polyhomogeneous $\mathit{\Psi}$DOs on $S^2M$, with leading symbols
\beq
\begin{split}
\sigma_L\Big(\projdiffperp\Big)(x,\xi)K&=\Pi_\xi^\perp K\Pi_\xi^\perp,\\
\sigma_L\Big(\projconfdiffperp\Big)(x,\xi)K&=\Pi_\xi^\perp K\Pi_\xi^\perp-\frac{1}{n-1}\tr\big(\Pi_\xi^\perp K\big)\Pi_\xi^\perp.
\end{split}
\eeq
Explicitly the operators are given (up to $\Psi^{-\infty}$) by:
\[
\begin{split}
\projdiffperp&=\Id-\nabla^\odot\Big[\diverg\nabla^\odot\Big]^{-1}\diverg,\\
\projconfdiffperp&=\Big(\Id-\fracsm{1}{n}g_0\cdot\tr\Big)\bigg\{\Id-\nabla^\odot\Big[\big(\diverg-\fracsm{1}{n}d\circ\tr\big)\nabla^\odot\Big]^{-1}\Big(\diverg-\fracsm{1}{n}d\circ\tr\Big)\bigg\}.
\end{split}
\]
\end{proposition}
\begin{remark}
Here $\nabla^\odot$ means taking the covariant derivative on 1-forms, followed by symmetrization. The inverses are pseudodifferential parametrices of the following geometric elliptic partial differential operators of order 2:
\[
\begin{split}
\diverg\nabla^\odot&:\CinfStwoM\to\CinfStwoM,\\
\big(\diverg-\fracsm{1}{n}d\circ\tr\big)\nabla^\odot&:\CinfStwoM\to\CinfStwoM.
\end{split}
\]
\end{remark}
\begin{proof}
Let $k\in\CinfStwoM$ be given. By the characterizations
\[
\diff=\Big\{\nabla^\odot\omega\:\Big\vert\:\omega\in\Omega^1M\Big\}\quad\textrm{and}\quad\diffperp=\Big\{k\in\CinfStwoM\:\Big\vert\:\diverg_{g_0}k=0\Big\},
\]
\[
\conf=\Big\{\varphi g_0\:\Big\vert\:\varphi\in\Cinf{M}\Big\}\quad\textrm{and}\quad\conf^\perp=\Big\{k\in\CinfStwoM\:\Big\vert\:\tr_{g_0}k=0\Big\},
\]
projecting on the subspace is equivalent to solving the elliptic system
\beq
\begin{cases}
\nabla^\odot\omega+\big(k-\nabla^\odot\omega\big)=k\\
\diverg_{g_0}k-\diverg_{g_0}\nabla^\odot\omega=0,
\end{cases}
\eeq
for $\omega$, respectively solving the system
\beq
\begin{cases}
\nabla^\odot\omega+\varphi\cdot g_0+(k-\nabla^\odot\omega-\varphi\cdot g_0)=k,\\
\diverg_{g_0}k-\diverg_{g_0}\nabla^\odot\omega-d\varphi=0,\\
\tr_{g_0}k-\tr_{g_0}\nabla^\odot\omega-n\cdot\varphi=0,
\end{cases}
\eeq
for $\omega$ and $\varphi$ (where we have used $\diverg_{g_0}(\varphi\cdot g_0)=d\varphi$).  To see that for instance the differential operator $\diverg_{g_0}\nabla^\odot$ is elliptic, we note that the leading symbol is
\beq
\sigma_L\Big[\diverg_{g_0}\nabla^\odot\Big](x,\xi)=|\xi|^2\big(I+\Pi_\xi\big),
\eeq
and thus it has elliptic (and indeed positive definite) leading symbol. The pseudodifferential parametrix then has leading symbol
\beq
\sigma_L\Big[\big(\diverg_{g_0}\nabla^\odot\big)^{-1}\Big](x,\xi)=\frac{|\xi|^{-2}}{2}\big(I+\Pi_\xi^\perp\big).
\eeq
Some further computations along these lines easily lead to the claimed results.
\end{proof}

\section{Spectral theory for log-polyhomogeneous pseudodifferential operators}\label{logsection}
In order to be able to take the derivative in the $s$-parameter in the classical polyhomogeneous pseudodifferential families studied in the previous sections, one needs for even dimensional manifolds to extend the class to log-polyhomogenous operators, as defined by M. Lesch \cite{Le}.

Firstly we review the log-polyhomogeneous class on a smooth manifold $M^n$ (e.g. \cite{Le}, see also \cite{PS}). Good general references for the pseudodifferential calculus discussed here include \cite{Sh} and \cite{Gi}. As always $H^s(M, E)$ denotes the closure of $\Cinf{M,E}$ with respect to the Sobolev norm corresponding to $s\in\Reals$, where $E$ is the rank $r$ vector bundle in which we work.

The class of order $d$ symbols $\mathrm{S}^d(U,\Reals^r)$, for $d\in\Reals$ and $U\subseteq\Reals^n$ open, are the functions $q(x,\xi)$ in $\Cinf{T^*U,\Endo{\Reals^r}}$, which satisfy the basic estimates
\beq
\Big\vert\partial_{x}^{\alpha}\partial_{\xi}^{\beta}q(x,\xi)\Big\vert\leq
C_{\alpha,\beta,K}(1+|\xi|)^{d-|\beta|},\quad\textrm{for}\:\xi\in\Reals^n,\: x\in K\subseteq U\:\textrm{compact}.
\eeq

The class $\Psi^d(M,E)$ of order $d$ pseudodifferential operators consists of the linear operators
\[
Q:C^\infty(E)\to C^\infty(E)
\]
s.t. in a neighborhood with a trivialization onto $U\times\Reals^r$, for an open set $U\subseteq\Reals^n$,
\beq\label{defPsiDOs}
(Qf)(x)=\int_{\Reals^n}\int_U e^{i(x-y)\cdot\xi}q(x,\xi)f(y)dyd\xi,\quad x\in U,\quad f\in C_c^\infty(U),
\eeq
with $q\in\mathrm{S}^d(U,\Reals^r)$.

The class of classical (1-step) polyhomogeneous symbols $CL^{\alpha}(U,\Reals^r)$ of order $\alpha\in\Complex$ are the functions $q\in\Cinf{T^*U,\Endo{\Reals^r}}$ having a sequence of symbols indexed by $j\geq0$ with $q_{\alpha-j}\in\Cinf{T^*U,\Endo{\Reals^r}}$, each being homogeneous in $\xi$ of degree $\alpha-j$
\[
q_{\alpha-j}(x,t\xi)=t^{\alpha-j}q_{\alpha-j}(x,\xi),\quad\mathrm{for}\quad|t|\geq 1,\quad |\xi|\geq 1
\]
and obeying the asymptotic expansion
\[
q(x,\xi)\sim\sum_{j=0}^{\infty}q_{\alpha-j}(x,\xi)
\]
as $|\xi|\to\infty$, meaning that for each $N\in\Naturals$
\[
q(x,\xi)-\sum_{j=0}^{N-1}q_{d-j}(x,\xi)\in\mathrm{S}^{\Real{\alpha}-N}(U,\Reals^r)
\]

The class of \emph{log-polyhomogeneous} symbols, extending the classical (1-step) polyhomogeneous symbol class, is as follows.

\begin{definition}\label{defPsiDOClass}
For an open set $U\subseteq\Reals^n$, we denote by
$\mathrm{CS}^{d,k}(U,\Reals^r)$ the log-polyhomogeneous symbols of order $d$ and log-degree $k$, being the set of symbols $q\in\bigcap_{\varepsilon>0}\mathrm{S}^{d+\varepsilon}(U,\Reals^r)$, having an asymptotic expansion
\begin{align*}
&q(x,\xi)\sim\sum_{j=0}^{\infty}q_{d-j}(x,\xi),\quad\rm{where}\\
&q_{d-j}(x,\xi)=\sum_{l=0}^kq_{d-j,l}(x,\xi)\log^l[\xi],
\end{align*}
and the $q_{d-j,l}(x,\xi)$ are positively homogeneous of degree $d-j$ in $\xi$.\\
\indent{}
Here we introduced $[\xi]$, which is a strictly positive $C^\infty$ function in $\xi$, which has $[\xi]=|\xi|$ for $|\xi|\geq 1$.
\end{definition}
\begin{remarks}
\begin{itemize}
\item[]
\item[(1)] The asymptotic equivalence in this context means that for each $N\in\Naturals$, the $N$'th difference satisfies
\[
q(x,\xi)-\sum_{j=0}^{N-1}q_{d-j}(x,\xi)\in\bigcap_{\varepsilon>0}\mathrm{S}^{d+\varepsilon-N}(U,\Reals^r).
\]
\item[(2)] The condition on the $q_{d-j}$ means that they belong to the class of (matrix-valued) log-polyhomogeneous functions $\mathcal{P}^{d-j,k}(U,\Reals^r)$.
\item[(3)] An alternative to the use of the modification $[\xi]$, is to multiply each $q_{d-j}$ by a cut-off $\psi\in\Cinf{\Reals^n}$ s.t. $\psi(\xi)=0$ for $|\xi|\leq 1/4$ and $\psi(\xi)=1$ for $|\xi|\geq 1/2$ (as in \cite{Le})
\end{itemize}
\end{remarks}
In the case $k=0$ without any logarithms
\[
\mathrm{CS}^{d,0}(U,\Reals^r)=\mathrm{CS}^d(U,\Reals^r),
\]
the classical 1-step polyhomogeneous symbols.
\begin{definition}
We denote by $\mathrm{CL}^{d,k}(U,\Reals^r)$ the class of pseudodifferential
operators which can be written in the form (\ref{defPsiDOs}) with
symbol $q\in\mathrm{CS}^{d,k}(U,\Reals^r)$.
\end{definition}
It has been verified by M. Lesch (in \cite{Le}) that the usual rules of calculus hold for
properly supported operators, also including change of coordinates by
diffeomorphisms. This as usual allows the definition of the
corresponding classes
$\mathrm{CL}^{d,k}(M)$ on manifolds $M$ as well as
$\mathrm{CL}^{d,k}(M,E)$ for vector bundles $E$. The leading symbol map,
locally $\sigma^d_L(Q):=q_d$, is
well-defined on $\mathrm{CL}^{d,k}(M,E)$ with kernel $\mathrm{CL}^{d-1,k}(M,E)$.

\subsection{Holomorphic families of classical polyhomogeneous $\mathbf{\Psi}$DOs}
In this section we review the topic of holomorphic families of classical
polyhomogeneous pseudodifferential operators with the specific goal of
understanding how the setting in \cite{Ok4} fits with the
log-polyhomogeneous class $\mathrm{CL}^{d,k}(M)$ just defined
above, i.e. we want derivatives of the family to be in this new class. Such issues have been addressed earlier, i.e. see \cite{Le},
\cite{KV}, \cite{PS} and others, who dealt with
somewhat different settings. The exposition here is not meant to be complete, but serves to explain the reduction on the assumptions, in for instance \cite{PS}, to those in the present context.

The definition of holomorphic families of classical $\Psi$DOs will be
as follows in order to match the conditions in \cite{Ok4}.
\begin{definition}\label{OkHoloFam}
If we have a family of $\mathit{\Psi}$DOs
\[
Q_s:\Cinf{E}\to\Cinf{E}
\]
depending on a parameter $s\in\Omega\subseteq\Complex$, then we will
say that $Q_s$ constitutes a
holomorphic family of polyhomogeneous
  $\mathit{\Psi}$DOs if each $Q_s$ is (classical 1-step) polyhomogeneous of order $\mu(s)$, where
\[
\mu:\Omega\to\Complex
\]
is holomorphic, and on any local trivialization of $E$, $Q_s$ has
symbol $q(s,x,\xi)$ which is smooth in $(s,x,\xi)$, holomorphic in s,
and satisfies the estimates
\begin{equation}\label{HoloEstimates}
\bigg\vert\partial_x^\alpha\partial^\beta_\xi\bigg(q(s,x,\xi)-\sum_{j=0}^{N-1}q_{\mu(s)-j}(s,x,\xi)\bigg)\bigg\vert\leq
C_{\alpha,\beta}(1+|\xi|)^{\Real\mu(s)-N-1-|\beta|},
\end{equation}
uniformly for $x$ and $s$ in compact sets.
\end{definition}
Noting the absence of $\partial^\gamma_s$ in the above expression, the symbols in this class may a priori seem to be too weakly estimated
with respect to the needed $s$-derivatives. As is seen from
the proof of the following proposition, the use of Cauchy estimates remedies this.
\begin{proposition}
Let $Q_s$ be a family of polyhomogeneous $\mathit{\Psi}$DOs, which is holomorphic in the
sense of Definition \ref{OkHoloFam}, with corresponding symbols
$q(s,x,\xi)$. Then in a trivialization on $U$:
\begin{itemize}
\item[(1)] Each homogeneous term $q_j(s,x,\xi)$ of $q_s$ is
  holomorphic in $s\in\Omega$, with derivative $\partial_s q_j\in\mathcal{P}^{\mu(s)-j,1}(U)$.
\item[(2)] For each $s\in\Omega$, $\partial_s q$ is in
  $\mathrm{CS}^{\mu(s),1}(U,\Reals^r)$ with symbol expansion
\[
\partial_sq(s,x,\xi)\sim\sum_j\partial_sq_j(s,x,\xi).
\]
\item[(3)] As operators on fixed Sobolev spaces, the family has first derivatives
  $\partial_sQ_s\in\mathrm{CS}^{\mu(s),1}(M,E)$, i.e.
\[
\partial_sQ_s:H^{t}(M,E)\to
H^{t-\Real\mu(s)-\varepsilon}(M,E),\quad\forall\varepsilon>0.
\]
The leading symbol of this operator is the holomorphic derivative of
the family of leading symbols of $Q_s$.
\end{itemize}
\end{proposition}

\begin{proof}
To see the first assertion we apply the recursive recovery formulae
for classical polyhomogeneous $\Psi$DOs
\[
q_{d-j}(s,x,\xi)=\lim_{t\to\infty}t^{-(d-j)}\bigg[q(s,x,t\xi)-\sum_{k=0}^{j-1}q_{d-k}(s,x,t\xi)\bigg].
\]
Using the estimates (\ref{HoloEstimates}) the convergence here is seen
to be uniform for $x$, $\xi$ and $s$ in compact sets. This implies
inductively the analyticity of the $q_j$ terms. Using the homogeneity we
find that $\partial_s q_j(s,x,\xi)$ equals
\[
\Big[\partial_s
q_j\Big(s,x,\frac{\xi}{|\xi|}\Big)\Big]|\xi|^{\mu(s)-j}+\Big[(\partial_s\mu)q_j\Big(s,x,\frac{\xi}{|\xi|}\Big)\Big]|\xi|^{\mu(s)-j}\log|\xi|\in\mathcal{P}^{\mu(s)-j,1}.
\]

To check that $\partial_sq\in\mathrm{S^{\mu(s)+\varepsilon}}$ for
all $\varepsilon >0$, we apply Cauchy estimates, which for holomorphic
$f$ gives
\[
\big\vert
f^{(n)}(s)\big\vert\leq\frac{n!\sup_{|w-s_0|=\varepsilon}f(w)}{\varepsilon^n},\quad\textrm{if
}|s-s_0|<\varepsilon.
\]
With this and the locally uniform estimates (\ref{HoloEstimates}), we get that
\begin{equation}
\Big\vert\partial_x^\alpha\partial^\beta_\xi \partial_sq(s,x,\xi)\Big\vert\leq
C_{\alpha,\beta,\varepsilon}(1+|\xi|)^{\Real\mu(s)+\varepsilon-|\beta|},
\end{equation}
uniformly for $x$ and $s$ in compact sets, for any $\varepsilon>0$.

To verify the asymptotic expansion we use the definition of
polyhomogeneity and continuity in $s$ to obtain estimates for each $N$
and $\varepsilon >0$
\[
\bigg\vert\partial_x^\alpha\partial_\xi^\beta\Big\{q(s,x,\xi)-\sum_{j=0}^{N-1}q_j(s,x,\xi)\Big\}\bigg\vert\leq
C_{\alpha,\beta,N,\varepsilon}(1+|\xi|)^{\Real\mu(s)+\varepsilon-N-|\beta|}
\]
uniformly for $x$ and $s$ in compact sets. Applying again Cauchy
estimates shows that the $N$'th difference
\[
\partial_sq(s,x,\xi)-\sum_{j=0}^{N-1}\psi(\xi)\partial_sq_j(s,x,\xi)\in\mathrm{S}^{\Real\mu(s)+\varepsilon-N},
\]
for any $\varepsilon >0$, and this was precisely the claim. Standard arguments now show that the considerations for the symbols of the family lead to (3) in the proposition.
\end{proof}

\begin{remark}
It is implicit to this discussion that we use a multiplicative cut-off (in $\xi$, as in Definition \ref{defPsiDOClass}) on $q(s,x,\xi)$, which amounts to changing $q_s$ by a family of smoothing symbols.

By (\ref{HoloEstimates}) this modifies the family $\partial_sQ_s$ by a holomorphic family of smoothing operators. As is easily seen, the properties of having compact resolvent and semi-bounded $L^2$-spectrum are thus unchanged. Consequently the cut-off will not affect the spectral results given in the main Theorem \ref{spectrum} below.
\end{remark}

\subsection{Ellipticity, positivity and spectrum of log-polyhomogeneous $\mathbf{\Psi}$DOs}\label{EllPos}
The usual notion of ellipticity for
pseudodifferential operators is quite strong and \textit{does not} carry over naturally to the
log-polyhomogeneous classes $\mathrm{CL}^{d,k}(M,E)$. The reason for this is that the parametrix will
generally only be of class $\mathrm{CL}^{-d,k}(M,E)$ and \textit{not} in $\mathrm{CL}^{-d-\varepsilon,k}(M,E)$, whenever $\varepsilon>0$. The following weaker notion of hypoellipticity is more adequate. We will not need the general classes $\mathrm{HS}^{d,d_0}_{\rho,\delta}(U,\Reals^r)$ and describe here only what corresponds to $(\rho,\delta)=(1,0)$.

\begin{definition}[\cite{Sh}, 5.1]\label{defhypo}
For real numbers $d_0\leq d$, the subclass $\mathrm{HS}^{d,d_0}(U,\Reals^r)$ of hypoelliptic symbols is the set of (matrix-valued) functions
\[
q(x,\xi)\in C^\infty(U,\Endo{\Reals^r}),
\]
where $U\subseteq\Reals^n$ is an open set, such that
\begin{itemize}
\item[(1)] For any compact $K\subseteq U$ there exist constants $R$,
  $C_1$ and $C_2$ such that
\[
C_1|\xi|^{d_0}\leq|q(x,\xi)|\leq C_2|\xi|^d,\qquad |\xi|\geq R,\quad
x\in K.
\]
\item[(2)] For each compact set $K\subseteq U$ there exists a constant $R$ such that
\[
\Big\vert q^{-1}(x,\xi)\big[\partial_{\xi}^{\alpha}\partial_{x}^{\beta}q(x,\xi)\big]\Big\vert\leq C_{\alpha,\beta,K}|\xi|^{-|\alpha|},\qquad |\xi|\geq R,\quad
x\in K,
\]
for suitable $C_{\alpha,\beta, K}$.
\end{itemize}
\end{definition}
An account of the standard results on these symbols can be found in \cite[5.1]{Sh}. First of all (since $(\rho,\delta)=(1,0)$ implies $1-\rho\leq\delta<\rho$) the notion extends to (vector bundles over) manifolds. We have indeed $\mathrm{HS}^{d,d_0}(M,E)\subseteq\mathrm{S}^{d}(M,E)$, and we note that hypoellipticity is determined by the leading symbol in $\mathrm{S}^{d}(M,E)$. We see that operators in $\mathrm{HS}^{d,d_0}(M,E)$ map Sobolev spaces $\mathrm{H}(M,E)^{s}\to\mathrm{H}(M,E)^{s-d}$, for any $s\in\Reals$. A crucial feature of the hypoelliptic class is the existence of a parametrix, as follows.
\begin{theorem}[\cite{Sh}, Thm. 5.1]\label{parametrix}
Let $Q\in\mathrm{HS}^{d,d_0}(M,E)$. Then there exists an operator $P\in\mathrm{HS}^{-d_0,-d}(M,E)$ such that
\[
QP=I+R_1,\quad PQ=I+R_2,
\]
where $R_1, R_2\in\Psi^{-\infty}$.
\end{theorem}
From this we may immediately observe that the hypoelliptic class has a corresponding G\aa{}rding type inequality.
\begin{corollary}[A hypoelliptic G\aa{}rding inequality]\label{Gaarding}
If $Q\in\mathrm{HS}^{d,d_0}(M,E)$ then there exists a constant $C_{Q,d,d_0}$ such that
\[
\|f\|_{d_0}\leq C_{Q,d,d_0}\big(\|f\|_0+\|Q f\|_0\big),\quad\textrm{for}\quad f\in\mathrm{H}^{d_0}(M,E).
\]
\end{corollary}
\begin{proof}
Let $f\in\mathrm{H}^{d_0}(M,E)$, apply Theorem \ref{parametrix} and use boundedness between Sobolev spaces to estimate as follows
\begin{align*}
\|f\|_{d_0}&=\|PQf\|_{d_0}+\|R_1f\|_{d_0}\\
&\leq C_Q\|Qf\|_{d_0-d}+C_{Q}'\|f\|_0\\
&\leq C_{Q,d,d_0}\big(\|f\|_0+\|Qf\|_0\big).
\end{align*}
\end{proof}

We can now define the correct generalizations of the notions of ellipticity and positive symbol to log-homogeneous operators.

\begin{definitions}
\begin{itemize}
\item[]
\item[(1)]The class of hypoelliptic symbols with positive leading symbols, denoted
\[
\mathrm{HS}_+^{d,d_0}(M,E)
\]
are those symbols in $\mathrm{HS}^{d,d_0}(M,E)$ having a leading symbol representative $\sigma_L^d$ and a constant $R$ such that
\[
\sigma_L^d(x,\xi)\textrm{ is strictly positive in }\big(\Endo{E_x},\langle\cdot,\cdot\rangle_x\big)\quad\textrm{for}\quad|\xi|\geq R.
\]
\item[(2)]The class of log-polyhomogeneous symbols that are hypoelliptic with positive leading symbols is
\[
\mathrm{HCL}_+^{d,k}(M,E):=\mathrm{CL}^{d,k}(M,E)\cap\bigcap_{\varepsilon>0}\mathrm{HS}_+^{d+\varepsilon,d}
\]
\end{itemize}
\end{definitions}
\begin{remark}
As mentioned in the introduction there is an important point here, in that the leading symbol of a log-polyhomogeneous operator includes every log-power of terms with highest $|\xi|$-power. Hence the notion of positive leading symbol \textit{does not} imply hypoellipticity. Namely the term with the highest log-degree may still be singular (while not negative), which is the case in the application we have in mind.
\end{remark}
The basic spectral results needed for dealing with the stability operators (i.e. $L^2$-Hessians) encountered here, are the
following results, quite analoguous to for instance Lemmas 1.6.3-1.6.4 in Gilkey's book \cite{Gi}, stated there only for partial differential operators with smooth coefficients.
\begin{theorem}\label{spectrum}
On a closed manifold $M$, suppose $Q\in\mathrm{HS}_+^{d,d_0}(M,E)$ with $0<d_0\leq d<d_0+1$ is
symmetric on the domain $\Cinf{M,E}$. Then $Q$ has a real discrete
spectrum, consisting of finite multiplicity eigenvalues $\{\lambda_k\}$, with $|\lambda_k|\to\infty$, and which is bounded from below, i.e. there exists a constant $c>0$ such
that $\mathrm{spec}(Q)\subseteq [-c,\infty)$.
\end{theorem}
\begin{corollary}
Theorem \ref{spectrum} applies to $\mathrm{HCL}_+^{d,k}(M,E)$ with $d>0$ by fixing an arbitrary $0<\varepsilon<1$ and taking $(d+\varepsilon,d)$ as the bi-degree. In particular operators in this class can have at most a finite number of negative eigenvalues.
\end{corollary}
Before proving this theorem, there is a small interlude to show that we can take square roots of positive leading symbol hypoelliptics, as needed below in the proof of Theorem \ref{spectrum}
\begin{lemma}[Square root lemma]\label{squareroot}
Let $d_0\geq 1$ and fix the notation that $\sqrt{\cdot}$ means using a partition of unity to construct an operator with leading symbol being the square root of the original leading symbol. Then we have
\beq
\sqrt{\mathrm{HS}_+^{d,d_0}(M,E)}\subseteq\mathrm{HS}^{\frac{d}{2},\frac{d_0}{2}}(M,E)
\eeq
\end{lemma}
\begin{proof}
Denoting the leading symbol by $\sigma_L$, Property (1) is clear since $|\sigma_L|=|\sigma_L^{1/2}|^2$, viz. by the C*-identity for $r\times r$-matrices.

For (2) we see accordingly
\[
\Big(\partial_x^\beta\partial_\xi^\alpha\sigma_L^{1/2}\Big)\sigma_L^{-1/2}=\sum_{\substack{|\gamma|\leq|\alpha|\\|\delta|\leq|\beta|}}C_{\alpha,\beta,\gamma,\delta}\:\sigma_L^{-(|\alpha|+|\beta|-|\gamma|-|\delta|)}\:\Big(\partial_\xi^\gamma\partial_x^\delta\sigma_L\Big)\sigma_L^{-1},
\]
from which the required estimates $|\cdot|\leq C_{\alpha,\beta}(1+|\xi|)^{-|\alpha|}$ follow, using $d_0\geq1$ and the assumption that $\sigma_L\in\mathrm{HS}_+^{d,d_0}(M,E)$.
\end{proof}

\begin{proof}[Proof of Theorem \ref{spectrum}]
Discreteness of the spectrum follows from Rellich's lemma on compactness of the resolvent on the compact manifold $M$, using the parametrix from Theorem \ref{parametrix}.

For semi-boundedness we may assume that $d_0>1$, or else take integer powers of the operator.
We construct an operator $A_0$ as above in the square root lemma (Lemma \ref{squareroot}). Letting $A:=A_0^*A_0$ we have
\[
\langle Af,f\rangle_0=\langle A_0f,A_0f\rangle_0,\quad\textrm{and}\quad Q-A\in\mathrm{S}^{d-1}.
\]
Now
\[
\langle Qf,f\rangle_0=\langle(Q-A)f,f\rangle_0+\langle Af,f\rangle_0,
\]
for $f\in\Cinf{M}$. The first term is estimated by
\[
\begin{split}
\vert\langle(Q-A)f,f\rangle_0\vert&\leq C\|f\|_{d_0/2}\|(Q-A)f\|_{-d_0/2}\\
&\leq C\|f\|_{d_0/2}\|f\|_{d-1-d_0/2}.
\end{split}
\]
By Lemma \ref{squareroot} we can apply G\aa{}rding's inequality, Corollary \ref{Gaarding}.
\beq
\|f\|_{d_0/2}^2\leq C\big(\|f\|_0^2+\|A_0f\|_0^2\big).
\eeq
For any $\delta>0$ we estimate, using Sobolev interpolation for $0<d-1-d_0/2<d_0/2$, since we are assuming $d_0>1$ and $0<d_0\leq d<d_0+1$,
\beq
\begin{split}
\big\vert\langle(Q-A)f,f\rangle_0\big\vert&\leq C\|f\|_{d_0/2}\|f\|_{d-1-d_0/2}
\leq\delta\|f\|_{d_0/2}^2+C(\delta)\|f\|_{d_0/2}\|f\|_0\\
&\leq2\delta\|f\|_{d_0/2}^2+\widetilde{C}(\delta)\|f\|_0^2\\
&\leq2C_\delta\|A_0f\|_0^2+C_2(\delta)\|f\|_0^2.
\end{split}
\eeq
Choosing $\delta>0$ such that $2C_\varepsilon\delta\leq 1$ finally proves
\[
\langle Qf,f\rangle_0\geq\|A_0\|^2-\vert\langle(Q-A)f,f\rangle_0\vert\geq -C_2(\delta)\|f\|_0^2.
\]
\end{proof}

\section{The stability operator for the zeta function of the Dirac operator: Proof of Theorem \ref{DiracHess}}\label{hessiansDirac}
Before deriving the formula for the stability operator, we describe the recent work by K. Okikiolu and by Okikiolu-Wang on pseudodifferential stability operators (i.e. $L^2$-Hessians) of spectral zeta functions, and it's extension to Dirac type operators. The paper by Okikiolu \cite{Ok4} contains major results needed in the present paper, namely the detailed heat kernel analysis in the construction of the $L^2$-Hessian and explicit formulae for its leading symbol. The setup explained in the introduction leads naturally to the generalization of the stability operator (i.e. $L^2$-Hessian) calculus to the spinor case, and hence may be used for the proof of Theorem \ref{DiracHess}.

Initially, we shall need to recall a few facts on the structure of the set $\Metr{M}$ of Riemannian metrics on the manifold $M$. First it is clear that this set is always a non-empty, convex (positive) cone inside the set $\Cinf{S^2M}$ of smooth symmetric covariant two-tensors on $M$. In fact it can be given a structure as an infinite-dimensional (Riemannian) manifold in various ways (see e.g. \cite{Eb} and \cite{FG}).

The Sobolev spaces can be defined as follows
\begin{definition}
$H^r(S^2M)$ is the completion of $\Cinf{S^2M}$ with respect to the norm $\Vert\cdot\Vert_r$ defined as
\[
\Vert k\Vert_r=\sum_{s\leq r}\DIP{\nabla^s k}{\nabla^s k}_g^{1/2},\quad k\in\Cinf{S^2M},
\]
where $g$ is a fixed ground metric.
\end{definition}
Smoothness of the functional allows the stability operator (i.e. $L^2$-Hessian) to be defined rigorously. On $\Metr{M}$ we take as smooth topology the Fréchet space topology coming from the collection of norms $\Vert\cdot\Vert_r$. If $r>m/2$ we denote by $\Metr{M}^r$ the closure of $\Metr{M}$ in $H^r(S^2M)$. The Sobolev theorem shows that then $H^r(S^2M)$ is contained in the set of continuous sections of $S^2M$. If $K\subseteq\Complex$ is compact and $\rm{Hol}(K)$ is the set of holomorphic functions on (open sets containing) $K$, given the supremum norm, then one has the following result, where $\mathcal{Z}_P(\cdot)$ is as in (\ref{modifiedzeta}) with $P$ satisfying these analytic and naturality assumptions.
\begin{proposition}[\cite{Ok4}]
Suppose $k,r\in\Naturals$ with $r>k+n/2+1$, and suppose $K\subseteq\{s\in\Complex\vert\Real s>k+n/2+1-r\}$. Then the map $g\to\mathcal{Z}_P(\cdot)$ is a $k$-times continuously differentiable map from $H^r(S^2M)$ to $\rm{Hol}(K)$.
\end{proposition}

The $L^2$-Hessian of the modified zeta function $\mathcal{Z}$ evaluated at
$k\in\CinfStwoM$ is
\beq
\Hess{\mathcal{Z}(P_g,s)(k,k)}=\frac{d^2}{dt^2}_{\big\vert t=0}\mathcal{Z}(P_{g+tk},s).
\eeq
Let $P$ be a second-order operator
(i.e. from now on $2l=2$). Assume that the structure group $H$ is $\Oof{n}$ or
$\SO{n}$. If this is the case and we require the analytic and
naturality assumptions, then $P$ is a \emph{geometric Laplace-type
  operator}. If we furthermore impose that $P$ has stable kernel
in a neighborhood of the ground metric, K. Okikiolu has recently proved:

\begin{theorem}[\cite{Ok4}]\label{OkHessTHM1}
For $s\in\Complex$, there exists a unique symmetric pseudodifferential
operator $T_s(P)$,
\[
T_s:\CinfStwoM\to \CinfStwoM,
\]
such that
\begin{equation}
\Hess\mathcal{Z}(s)(k,k)=\DIP{k}{T_s k}_g.
\end{equation}
The operator $T_s$ is analytic in $s$. For $s\notin n/2+\Naturals^+$ there exist polyhomogeneous pseudodifferential operators $U_s$ and $V_s$ of
degrees $n-2s$ and 2, respectively, such that $T_s=U_s+V_s$. The
operators $U_s$ and $T_s$ are meromorphic in $s$ with simple poles
located in $n/2+\Naturals^{+}$. For general $s$, the
polyhomogeneous symbol expansion of $U_s$ is computable from the
complete symbol of the operator $P$. In particular, there is a simple
algorithm to compute the term $u_s$ of homogeneity
$n-2s$. Furthermore, we can differentiate in $s$ to obtain
\begin{equation}
\Hess\Big(\frac{d}{ds}\Big)^l\mathcal{Z}(s)(k,k)=\BigDIP{k}{\Big(\frac{d}{ds}\Big)^lT_s k}_g,
\end{equation}
and the principal symbol of $(d/ds)^lU_s$ is equal to the leading
order term of $(d/ds)^lu_s$ (provided it does not vanish identically).

The symbol $u_s$ can be computed as follows. Let $x\in M$ and take coordinates on $M$ which are orthonormal at $x$, and take a local trivialization of $E$ on a neighborhood of $x$ to obtain coordinates for $E$. Suppose that in these coordinates
the operator $P'=\frac{d}{dt}_{\big\vert t=0}P_{g+tk}$ is given at the point $x$ by
\begin{equation}
P'=\sum_{\alpha, \beta, i, j}A^{ij}_{\alpha,\beta}\big(\partial^\alpha_wk_{ij}\big)\partial^\beta_w,
\end{equation}
where $\alpha$ and $\beta$ are multi-indices and
$A^{ij}_{\alpha,\beta}$ is an $N\times N$ real-valued matrix (for each
choice of the indices
$i,j,\alpha, \beta$). Set
\beq
C(s)=\Big(\frac{1}{4\pi}\Big)^{n/2}\frac{\Gamma(-S+1)^2}{\Gamma(-2S+2)},\quad\textrm{where}\quad S=s-\frac{n}{2}.
\eeq
Then at $(x,\xi)\in\CoTM$, the value of $u_s(x,\xi)\in\Endo{S^2M}_x$ is given by
\[
(u_s(x,\xi)k)_{ij}=V^{\frac{2s-n}{n}}C(s)
\sum_{\substack{
  |\alpha|+|\beta|=2\\
  |\gamma|+|\delta|=2\\
  k,l}}u_s(\dalpha,\dbeta,\dgamma,\ddelta,x,\xi)\tr(A\ijab(x)A\klgd(x))k_{kl},
\]
where the terms $u_s(\dalpha,\dbeta,\dgamma,\ddelta,x,\xi)$ are given as follows ($I$ denotes identity on $E_x$):
\begin{align*}
u_s(\dsubj\dsubk,I,\dsubp\dsubq,I,x,\xi)=&\ 4(4S^2-1)\xij\xik\xip\xiq\xinorm{n-2s-4},\\
u_s(\dsubj,\dsubk,\dsubp\dsubq,I,x,\xi)=&-2(4S^2-1)\xij\xik\xip\xiq\xinorm{n-2s-4},\\
u_s(\dsubj,\dsubk,\dsubp,\dsubq,x,\xi)=&\ (4S^2+2S-2)\xij\xik\xip\xiq\xinorm{n-2s-4}-(2S-1)\delta_{kq}\xij\xip\xinorm{n-2s-2},\\
u_s(I,\dsubj\dsubk,\dsubp\dsubq,I,x,\xi)=&\ (4S^2-2S)\xij\xik\xip\xiq\xinorm{n-2s-4}+(2S-1)\delta_{jk}\xip\xiq\xinorm{n-2s-2},\\
u_s(I,\dsubj\dsubk,\dsubp,\dsubq,x,\xi)=&-(2S^2+S-1)\xij\xik\xip\xiq\xinorm{n-2s-4}\\&+\Big(S-\frac{1}{2}\Big)(-\delta_{jk}\xip\xiq+\delta_{jq}\xik\xip+\delta_{kq}\xij\xip)\xinorm{n-2s-2},\\
u_s(I,\dsubj\dsubk,I,\dsubp\dsubq,x,\xi)=&\
(S^2+S)\xij\xik\xip\xiq\xinorm{n-2s-4}\\&+\frac{1}{2}(S-1)\delta_{jk}\xip\xiq\xinorm{n-2s-2}+\frac{1}{2}(S-1)\xij\xik\delta_{pq}\xinorm{n-2s-2}\\&-\frac{S}{2}(\delta_{jp}\xik\xiq+\delta_{kq}\xij\xip+\delta_{jq}\xik\xip+\delta_{kp\xij\xiq})\xinorm{n-2s-2}\\&\
\frac{1}{4}(\delta_{jk}\delta_{pq}+\delta_{jp}\delta_{kq}+\delta_{jq}\delta_{kp})\xinorm{n-2s}.
\end{align*}
\end{theorem}

We may more generally let $E$ be a tensor-spinor bundle, as in Section \ref{SummaryAndResults}. The adaptation to this case follows immediately from the details of the proof in \cite{Ok4}. Note that the usual heat kernel estimates are applicable here, in this broader class of generalized Laplacians (see e.g. Chapter 2 in \cite{BGV}). As explained in Section \ref{Section:Diracoperators} and Appendix \ref{BGAppendix}, the Bourguignon-Gauduchon transform allows us to consider a family of operators in a fixed vector bundle, and as in \cite{Ok4}, the main term in the analysis is
\[
\int\int_{u+v<1}(u+v)^sP_k'e^{-uP}P_k'e^{-vP}dudv,
\]
where $P_k'$ is the derivative at $t=0$ of the operator $P$ along a curve of metrics $g+tk$. The proof then proceeds completely analogously, noting that the arguments are local after the point where one trivializes the bundle $E$ and applies normal geodesic coordinates locally on $M$. 

On the local level the analysis relies on the same standard estimates and analysis in Euclidean space, for elliptic operators with matrix-valued symbols, as in the non-spinorial Laplacian case. Note that the Clifford multiplication in the definition of the Atiyah-Singer-Dirac operator in (\ref{DiracDef}) just comprises the choice of one specific endomorphism field, while the proof in \cite{Ok4} works for any such globally consistent construction involving a fiberwise action of the endomorphism bundle of a vector bundle (e.g. the spinor bundle).

\begin{remark}Concerning the gauged Dirac operators, we note that one could try
incorporating square roots of (natural) symmetric endomorphisms into
the concept of natural geometric tensor-spinor operators. Here,
this is not necessary, since, as can easily be verified, both $P$, $P'$ and $P''$  (i.e. differentiating and evaluating at zero perturbation) are in this case natural (see Theorem \ref{varDirac}), which is all that is needed for the proof in \cite{Ok4} to go through. We thus have the following corollary.
\end{remark}

\begin{corollary}\label{OkikioluForDirac}
On the closed Riemannian spin manifold $M$, let $P$ be a second order differential operator in $\Cinf{E}$ satisfying the Analytical assumptions
\ref{analassump} and Naturality assumptions \ref{naturassump}. Then
the above Theorem \ref{OkHessTHM1} and Theorem \ref{OkHessTHM1} hold
for this operator as well. In fact it also applies to the
Bourguignon-Gauduchon gauge transformed squares of Dirac operators, which are not natural in that particular sense.
\end{corollary}

As mentioned above the endomorphisms in the trivialization of $E$, represented by
matrices $A_{\alpha\beta}^{ij}$ in Theorem \ref{OkHessTHM1}, in the spin case are induced by repeated Clifford multiplication. A crucial step in the following will
be to find and express the traces of these endomorphisms in a manageable way.

To prove Theorem \ref{DiracHess}, using Corollary \ref{OkikioluForDirac}, we find the local coordinates expression of the derivative of the Dirac operator, using the gauge transform of Bourguignon-Gauduchon and first variation formula in Equation \ref{BGVarIntro}.
\begin{remark}
To give results of the type in Theorem \ref{DiracHess}, i.e. with the leading term as stated, there is always the qualifier that the right hand side does not vanish identically. Note that the zeros of the $\Gamma$-factor are all the region in $\Real
s\geq n/2+1$, and thus there are surely no zeros in the half-plane $\Real s<n/2-1$ considered here.
\end{remark}

We will show the following result on the local form of the first variation of the Dirac operator
\begin{proposition}
In a trivialization of $E$ and in normal geodesic coordinates
$\{x_i\}$ in a neighborhood of $x\in U\subseteq M$, i.e. $e_i=\partial_{i\vert x}$ gives an orthonormal basis, we have
\beq\label{localform}
(\Dirac^2)'_{\vert x}=\sum_{i,j,k,l}D_{kl}^{ij}k_{ij}\dsubk\dsubl+\sum_{i,j,k,l}A_{kl}^{ij}(\dsubk
k_{ij})\dsubl+\sum_{i,j,k,l}B_{kl}^{ij}(\dsubk\dsubl k_{ij})+\textrm{DOTs},
\eeq
where the endomorphisms of the fiber $(\Sigma M_\gamma)_{x}$ have the matrix coefficients
\beq\label{DAB}
\begin{split}
D_{kl}^{ij}&=\delta_{ik}\delta_{jl}Id,\\
A_{kl}^{ij}&=-\fracsm{1}{2}\big(\deltakl\deltaij Id - \deltalj\deltaik
Id +\deltalj e_ke_i\big),\\
B_{kl}^{ij}&=\fracsm{1}{4}\big(\deltaij e_le_k-\deltaik e_le_j\big),
\end{split}
\eeq
all as Clifford multiplications on the spinor fields.
\end{proposition}
\begin{remark}
Here and later in this derivation, we disregard terms that are not of total degree 2 in the number of derivatives falling on $k$ plus the number of spinor derivatives. In a mnemonic:
\begin{center}
"\ldots" = DOTs = Degree Other than Two's
\end{center}
\end{remark}
\begin{proof}
Recall the local representation (\ref{spinnablalocal}). From this and Theorem \ref{varDirac} we see that
\begin{align*}
(\Dirac^2)'\psi&=\Dirac'\Dirac\psi+\Dirac\Dirac'\psi\\
&=-\frac{1}{2}\sum_{i,j}e_i\widetilde{\nabla}_{K_{g}(e_i)}\Big(e_j\widetilde{\nabla}_{e_j}\psi\Big)+\frac{1}{4}\sum_j\big[d(\tr_g k)-\diverg_g
    k\big]_{\cdot\gamma}e_j\widetilde{\nabla}_{e_j}\psi\\
&\quad -\frac{1}{2}\sum_{i,j}e_j\widetilde{\nabla}_{e_j}\Big(e_i\widetilde{\nabla}_{K_{g}(e_i)}\psi\Big)+\frac{1}{4}\sum_je_j\widetilde{\nabla}_{e_j}\big[d(\tr_g k)-\diverg_g
    k\big]_{\cdot\gamma}\psi\\
&=-\frac{1}{2}\sum_{i,j}e_ie_j\widetilde{\nabla}_{K_{g}(e_i)}\widetilde{\nabla}_{e_j}\psi-\frac{1}{2}\sum_{i,j}e_i\Big(\nabla_{K_{g}(e_i)}e_j\Big)\widetilde{\nabla}_{e_j}\psi\\
&\quad +\frac{1}{4}\sum_j\big[d(\tr_g k)-\diverg_g k\big]_{\cdot\gamma}e_j\widetilde{\nabla}_{e_j}\psi
-\frac{1}{2}\sum_{i,j}\Big(e_j\nabla_{e_j}e_i\Big)\widetilde{\nabla}_{K_{g}(e_i)}\psi\\
&\quad
-\frac{1}{2}\sum_{i,j}e_je_i\widetilde{\nabla}_{e_j}\widetilde{\nabla}_{K_{g}(e_i)}\psi+\frac{1}{4}\sum_j\Big(e_j\nabla_{e_j}\big[d(\tr_gk)-\diverg_gk\big]_{\cdot\gamma}\Big)\psi\\
&\quad +\frac{1}{4}\sum_je_j\big[d(\tr_gk)-\diverg_gk\big]_{\cdot\gamma}\widetilde{\nabla}_{e_j}\psi
\end{align*}
We rewrite this using the Clifford identities
\[
Xe_i+e_iX=-2g(X,e_i),\quad X\in\Cinf{TM},
\]
to see that
\begin{align*}
&\frac{1}{4}\sum_j\big[d(\tr_g k)-\diverg_g
k\big]_{\cdot\gamma}e_j\widetilde{\nabla}_{e_j}\psi+
\frac{1}{4}\sum_je_j\big[d(\tr_gk)-\diverg_gk\big]_{\cdot\gamma}\widetilde{\nabla}_{e_j}\psi\\
&\quad=-\frac{1}{2}\sum_j\big[d(\tr_g
k)-\diverg_g k\big]_j\widetilde{\nabla}_{e_j}\psi,
\end{align*}
and pick up some DOTs, namely the terms
\beq
-\frac{1}{2}\sum_{i,j}e_i\Big(\nabla_{K_{g}(e_i)}e_j\Big)\widetilde{\nabla}_{e_j}\psi,\quad\textrm{and}\quad
-\frac{1}{2}\sum_{i,j}\Big(e_j\nabla_{e_j}e_i\Big)\widetilde{\nabla}_{K_{g}(e_i)}\psi,
\eeq
to get the expression
\begin{align*}
(\Dirac^2)'\psi&=
-\frac{1}{2}\sum_{i,j}e_ie_j\Big(\widetilde{\nabla}_{K_{g}(e_i)}\widetilde{\nabla}_{e_j}+\widetilde{\nabla}_{e_i}\widetilde{\nabla}_{K_g(e_j)}\Big)\psi-\frac{1}{2}\sum_j\big[d(\tr_g
k)-\diverg_g k\big]_j\widetilde{\nabla}_{e_j}\psi\\
&\quad +\frac{1}{4}\sum_j\Big[e_j\nabla_{e_j}\big(d(\tr_g
k)-\diverg_g k\big)\Big]_{\cdot\gamma}\psi+DOTs.
\end{align*}
Everything will now be expressed in the normal geodesic local coordinates
$\{x_i\}$, using the frame of coordinate vector fields $\{\dsubi\}$. By Gram-Schmidt orthonormalization (which is a smooth
process) on the frame field $\{\dsubi\}$, we also get the smooth local orthonormal frame field $\{e_i\}$ which
at $x$ coincides with $\{\dsubi\}$, since this is already orthonormal at $x$. Thus also $g_{ij\vert x}=\delta_{ij}$. We insert this special orthonormal frame in all the above formulae and denote by $\varphi$ the field of invertible linear transitions
\beq
e_i=\varphi_i^l\dsubl.
\eeq
Note that $\varphi_{i\vert x}^l=\delta_i^j$ and that $\varphi$ only depends on the ground metric $g$ and the chart, but not on
the tangent field $k$.
\newline\indent{}
In local coordinates we have
\begin{align}
&K_g(\dsubi)=g^{lk}k_{ki}\dsubl,\\
&K_g(e_i)=g^{pk}\varphi_i^lk_{kl}\dsubp,\\
&(\delta k)_j=g^{ik}k_{ij,k}=g_{ik}\Big(\dsubk
k_{ij}-\Gamma^l_{ki}k_{lj}-\Gamma^l_{kj}k_{il}\Big),\\
&\big(d(\tr_gk)\big)_j=\dsubj\Big(g^{ik}k_{ik}\Big)=g^{ik}\Big(\dsubj k_{ik}-\Gamma^l_{ji}k_{lk}-\Gamma^l_{jk}k_{il}\Big),
\end{align}
using the convention of Einstein summation. The last equality follows easily for instance from the fact that the Levi-Civita connection,
extended to tensor fields, acts as the exterior derivative on functions, while commuting with musical isomorphisms and tensor contractions.
\newline\indent{}
Using the above with (\ref{spinnablalocal}) to write everything locally, and picking
up some more DOTs, we see
\begin{align*}
(\Dirac^2)'\psi
&=-\frac{1}{2}\sum_{i,j}e_ie_j\Big[\fracsm{\partial}{\partial
  K_g(e_i)}\fracsm{\partial}{\partial e_j}+\fracsm{\partial}{\partial
  e_i}\fracsm{\partial}{\partial K_g(e_j)}\Big]\psi\\
&\quad -\frac{1}{2}\sum_{i,j,k}g^{ik}\big[(\dsubj
k_{ik})\fracsm{\partial}{\partial e_j}-(\dsubk k_{ij})\fracsm{\partial}{\partial e_j}\big]\psi\\
&\quad +\frac{1}{4}\sum_{i,j,k,l}\big[e_je_l\fracsm{\partial}{\partial
  e_j}\big\{g^{ik}\big(\dsubl k_{ik}-\dsubk k_{il}\big)\}\big]\psi+DOTs.
\end{align*}
To obtain the canonical form we rewrite, using that $\varphi$-derivatives are DOTs.
\begin{align*}
\frac{\partial}{\partial
 K_g(e_i)}\frac{\partial}{\partial e_j}&=
g^{pk}\varphi^l_ik_{kl}\dsubp
 (\varphi_j^q\dsubq)=g^{pk}\varphi^l_ik_{kl}\varphi_j^q\dsubp
 \dsubq+DOTs,
\end{align*}
\begin{align*}
\frac{\partial}{\partial
 e_i}\frac{\partial}{\partial K_g(e_j)}&=
g^{pk}\varphi^l_j\Big[\big(\fracsm{\partial}{\partial
 e_i}k_{kl}\big)\dsubp+k_{kl}\fracsm{\partial}{\partial
 e_i}\dsubp\Big]+DOTs.\\
\end{align*}
Evaluating at $x$, the center of the normal geodesic coordinates, we
get
\begin{align*}
&\sum_{i,j}e_ie_j\Big[\fracsm{\partial}{\partial
  K_g(e_i)}\fracsm{\partial}{\partial e_j}+\fracsm{\partial}{\partial
  e_i}\fracsm{\partial}{\partial K_g(e_j)}\Big]_{\big\vert x}\\
&\qquad\qquad\qquad=\sum_{i,j,k}e_ie_j\big(k_{ki}\dsubk\dsubj+k_{kj}\dsubi\dsubk\big)
+\sum_{i,j,k}e_ie_j\big(\dsubi k_{kj}\big)\dsubk\\
&\qquad\qquad\qquad=-2\sum_{i,j}k_{ij}\dsubi\dsubj+\sum_{i,j,k}e_ie_j\big(\dsubi k_{kj}\big)\dsubk.
\end{align*}
The final expression is thus
\begin{align*}
(\Dirac^2)'_{\big\vert
  x}&=\sum_{i,j}k_{ij}\dsubi\dsubj-\frac{1}{2}\sum_{i,j,k}e_ie_j\big(\dsubi
  k_{kj}\big)\dsubk-\frac{1}{2}\sum_{i,j}\big[\dsubj k_{ii}-\dsubi k_{ij}\big]\dsubj\\
&\quad+\frac{1}{4}\sum_{i,j,k}e_je_k\big[\dsubj\dsubk k_{ii}-\dsubi\dsubj k_{ik}\big]
+DOTs.
\end{align*}
From this formula, the above expressions for $A$, $B$ and $D$ can now be seen directly.
\end{proof}

For deriving the leading symbol of the stability operator, we fix some convenient notation, similarly to what has proven convenient in \cite{OW}.
\beq
(k\cdot\xi)_i=\sum_jk_{ij}\xi_j,\quad \xi\cdot
k\cdot\xi=\sum_{i,j}k_{ij}\xi_i\xi_j,\quad
|k|=\bigg[\sum_{i,j}k_{ij}^2\bigg]^{1/2},\quad \tr k=\sum k_{ii}.
\eeq
As can easily be computed
\beq
\begin{split}
&\Big(K_g\Pi_\xi^\bot\Big)_{ij}=k_{ij}-|\xi|^{-2}\sum_kk_{ik}\xi_k\xi_j,\\
&\tr{(K_g\Pi^{\bot}_\xi)^2}=|k|^2-2|\xi|^{-2}|k\cdot\xi|^{2}+|\xi|^{-4}(\xi\cdot
k\cdot\xi)^2, \\
&(\tr{K_g\Pi^{\bot}_\xi})^2=|\xi|^{-4}(\xi\cdot
k\cdot\xi)^2-2|\xi|^{-2}(\xi\cdot k\cdot\xi)(\tr k)+(\tr k)^2.
\end{split}
\eeq
Using the local form of the infinitesimal variation (\ref{localform})
one may also define the coefficient symbols, formally
substituting $\xi_i$ for each $\dsubi$.
\beq
\sigma^{(2)}=Id\sum_{i,j}k_{ij}\xi_i\xi_j,\quad
\sigma^{(1)}=\sum_{i,j,k,l}A^{ij}_{kl}\xi_k\xi_lk_{ij},\quad
\sigma^{(0)}=\sum_{i,j,k,l}B^{ij}_{kl}\xi_k\xi_lk_{ij}
\eeq
By the coefficient components we mean
\beq
\sigma^{(1)}_{kl}=\sum_{i,j}A^{ij}_{kl}k_{ij},\quad
\sigma^{(0)}_{kl}=\sum_{i,j}B^{ij}_{kl}k_{ij}.
\eeq

For the sake of book-keeping, express $u_s$ as a sum
\[
u_s=C(m,s)\Big(u_s^{(1)}+u_s^{(2)}+u_s^{(3)}+u_s^{(4)}\Big),
\]
where the individual terms can be calculated as \cite{OW}
\beq
\begin{split}
\big\langle k,u_s^{(1)}(x,\xi)\,k\big\rangle=
&\ \Big(S^2-\frac{1}{4}\Big)|\xi|^{n-2s-4}\tr\Big(\sigma^{(2)}-2\sigma^{(1)}+4\sigma^{(0)}\Big)^2,\\
\big\langle k,u_s^{(2)}(x,\xi)\,k\big\rangle=
&\ \Big(2S-1\Big)|\xi|^{n-2s-4}\tr\bigg(\big(\sigma^{(1)}\big)^2-|\xi|^2\sum_j\Big[\sum_i\xi_i\sigma_{ij}^{(1)}\Big]^2\bigg),\\
\big\langle k,u_s^{(3)}(x,\xi)\,k\big\rangle=
&\ \Big(2S-1\Big)|\xi|^{n-2s-4}\bigg\{|\xi|^2\tr\Big(\sum_{i,j}\sigma_{ij}^{(1)}\big[\xi_i(k\cdot\xi)_j+\xi_j(k\cdot\xi)_i\big]\Big)\\
&\ +(\xi\cdot k\cdot\xi)\tr\Big(-\sigma^{(1)}-2\sigma^{(0)}\Big)
+|\xi|^2(\tr k)\tr\Big(-\sigma^{(1)}+2\sigma^{(0)}\Big)\bigg\},\\
\big\langle k,u_s^{(4)}(x,\xi)\,k\big\rangle=
&\ \dim
E|\xi|^{n-2s-4}\Big(S-\frac{1}{2}\Big)\times\\
&\ \Big(-2|\xi|^2|k\cdot\xi|^2+(\xi\cdot
k\cdot\xi)^2+|\xi|^2(\xi\cdot k\cdot\xi)(\tr k)\Big)\\
&\ +\dim E |\xi|^{n-2s}\bigg(\frac{1}{2}\tr{(K_g\Pi^{\bot}_\xi)^2}+\frac{1}{4}(\tr{K_g\Pi^{\bot}_\xi})^2\bigg).
\end{split}
\eeq
The various quantities needed are found from Equations (\ref{DAB}).
\beq
\begin{split}
&\sigma^{(2)}=(\xi\cdot k\cdot\xi)\\
&\sigma^{(1)}=-\frac{1}{2}\Big((\tr k)|\xi|^2-(\xi\cdot k\cdot\xi)+\xi(k\cdot\xi)\Big)\\
&\sigma^{(0)}=\frac{1}{4}\big((\tr k)\xi\xi-\xi(k\cdot\xi)\big)\\
&\sigma^{(1)}_{kl}=-\frac{1}{2}\Big((\tr k)\delta_{kl}-k_{kl}+e_k\sum_ie_ik_{il}\Big).\\
\end{split}
\eeq
Note that here the expressions involve Clifford
multiplication. E.g. $\xi(k\cdot\xi)$ means Clifford multiplication by
the vector $k\cdot\xi$ followed by multiplication by $\xi$, two
operations that \emph{do not commute}, and note that $\xi\xi\neq |\xi|^2$.
To evaluate the four parts of the leading symbol, we need to take
traces in the fibers of the spin bundle. For this it is convenient to concentrate some technicalities into the following proposition.

\begin{proposition}[Trace of an endomorphism of repeated Clifford multiplications]\label{CliffTrace}
Let $\{e_i\}$ be an ONB of $TM_x$. Let $1\leq i_1\leq\dots\leq i_{2l}\leq n$ be numbers
such that $i_{\mu}\neq i_1$ for $\mu\neq 1$. Then
\begin{itemize}
\item[(1)] $\tr\bigg(\prod_{r=1}^{2l}e_{i_r}\bigg)=0$,
\end{itemize}
where the trace is taken in the fiber $\Sigma
M_x\simeq\Complex^{2^k}$, where the vectors act through projection in
the Clifford algebra, according to the identifications:
\begin{alignat*}{2}
&\CliffC_{2k} &\simeq\; &\mathrm{M}_{2^k}(\Complex)\\
&\CliffC_{2k+1} &\simeq\; &\mathrm{M}_{2^k}(\Complex)\oplus\mathrm{M}_{2^k}(\Complex).
\end{alignat*}
Furthermore denoting the metric on $T_xM$ by
$\langle\cdot,\cdot\rangle$ we have:
\begin{itemize}
\item[(2)] $\tr(ab)=-\dim E\langle a, b\rangle, \quad
  a,b\in T_xM$.
\item[(3)] $\tr(abcd)=\dim E \Big\{\langle a,
  b\rangle\langle c, d\rangle-\langle a,c\rangle\langle
  b,d\rangle+\langle a,d\rangle\langle b,c\rangle\Big\}, \quad
  a,b,c,d\in T_xM$.
\item[(4)] $\tr(abab)=\dim E\Big\{2\langle
  a,b\rangle^2-|a|^2|b|^2\Big\}$, as a special case of (3).
\end{itemize}
\end{proposition}
\begin{proof}
By the trace property $\tr(\varphi\circ \psi)=\tr(\psi\circ \varphi)$,
calculating in the Clifford algebra and using that the vector action
gives an algebra morphism:
\begin{align*}
 \tr\bigg(\prod_{r=1}^{2l}e_{i_r}\bigg)
&= \tr\bigg(e_1\prod_{r=2}^{2l}e_{i_r}\bigg) = (-1)^{2l-1}\tr\bigg(\prod_{r=1}^{2l}e_{i_r}\bigg)
= -\tr\bigg(\prod_{r=1}^{2l}e_{i_r}\bigg),
\end{align*}
which proves the first statement.
\newline\indent{}
Applying (1) with $l=1$ and expanding in the basis, we see
\[
\tr(ab)=\sum_{i,j}a^ib^j\tr(e_ie_j)=\sum_{i,j}\Big\{-(\dim E)\deltaij a^ib^j +(1-\deltaij)a^ib^j\tr(e_ie_j)\Big\}=\langle a,b\rangle.
\]
Continuing in this inclusion-exclusion fashion, we see that
\begin{align*}
\tr(e_ie_je_ke_l)&=-\deltaij\tr(e_ke_l)+(1-\deltaij)\deltaik\tr(e_ie_je_ke_l)\\
&\quad+(1-\deltaij)(1-\deltaik)\deltail\tr(e_ie_je_ke_l)\\
&\quad+(1-\deltaij)(1-\deltaik)(1-\deltail)\tr(e_ie_je_ke_l)\\
&=\dim E\Big\{\deltaij\deltajk-(1-\deltaij)\deltaik\deltajl+(1-\deltaij)(1-\deltaik)\deltail\deltajk+0\Big\}
\end{align*}
Thus expanding in the basis we get
\begin{align*}
\tr(abcd)&=\dim E\sum_{i,j,k,l}a^ib^jc^kd^l\Big\{\deltaij\deltakl-(1-\deltaij)\deltaik\deltajl+(1-\deltaij)(1-\deltaik)\deltail\deltajk\Big\}\\
&=\dim E \Big\{\langle a,
  b\rangle\langle c, d\rangle-\langle a,c\rangle\langle
  b,d\rangle+\langle a,d\rangle\langle b,c\rangle\Big\},
\end{align*}
which is (3).
\end{proof}
The proposition gives in particular the following useful formulae.
\begin{corollary}
\begin{align*}
&\tr(\xi\xi)\quad=-\dim E|\xi|^2,\\
&\tr(\xi\xi\xi\xi)=\dim E |\xi|^4,\\
&\tr(\xi(k\cdot\xi))=-(\xi\cdot k\cdot\xi),\\
&\tr(\xi(k\cdot\xi)\xi(k\cdot\xi))=2(\xi\cdot k\cdot\xi)^2-|\xi|^2|k\cdot\xi|^2,
\end{align*}
\end{corollary}

Finally we can apply the preceding to complete the proof of our main Theorem \ref{DiracHess}.
\begin{proof}[Proof of Theorem \ref{DiracHess}]
To find $u_s^{(1)}$, we evaluate the trace
\begin{align*}
&\tr\Big(\sigma^{(2)}-2\sigma^{(1)}+4\sigma^{(0)}\Big)^2\\
&=\tr\Big((\xi\cdot k\cdot\xi)+(\tr k)|\xi|^2-(\xi\cdot k\cdot\xi)+\xi(k\cdot\xi)+(\tr k)\xi\xi-\xi(k\cdot\xi)\Big)^2\\
&=\tr\Big((\tr k)^2|\xi|^4+(\tr k)^2\xi\xi\xi\xi+2|\xi|^2(\tr k)^2\xi\xi\Big)\\
&=0,
\end{align*}
thus giving $\big\langle k,u_s^{(1)}(x,\xi)\,k\big\rangle=0$.

For $u_s^{(2)}$ we calculate
\begin{align*}
\sum_j\bigg[\sum_i\xi_i\sigma_{ij}^{(1)}\bigg]^2&=\frac{1}{4}\sum_j\bigg[\sum_i\xi_i\Big((\tr k)\deltaij-k_{ij}+e_i\sum_ke_kk_{kj}\Big)\bigg]^2\\
&=\frac{1}{4}\sum_j\bigg[(\tr k)\xi_j-(k\cdot\xi)_j+\xi\Big(\sum_ke_kk_{kj}\Big)\bigg]\\
&=\frac{1}{4}\sum_j\bigg[(\tr k)^2\xi_j^2+(k\cdot\xi)^2_j+\xi\Big(\sum_ke_kk_{kj}\Big)\xi\Big(\sum_{k'}e_{k'}k_{k'j}\Big)\\
&\qquad -2(\tr k)\xi_j(k\cdot\xi)_j+2(\tr k)\xi\Big(\sum_k\xi_jk_{kj}\Big)
-2\xi(k\cdot\xi)_j\Big(\sum_ke_kk_{kj}\Big)
\bigg].
\end{align*}
Taking the trace gives
\begin{align*}
&\tr\sum_j\bigg[\sum_i\xi_i\sigma_{ij}^{(1)}\bigg]^2\\
&\quad=\frac{\dim E}{4}\bigg\{|\xi|^2(\tr
k)^2+|k\cdot\xi|^2+2\sum_j\bigg\langle\xi,\sum_ke_kk_{kj}\bigg\rangle^2
-|\xi|^2\sum_j\bigg\vert\sum_ke_kk_{kj}\bigg\vert^2\\
&\qquad\qquad\qquad\quad-2(\tr k)(\xi\cdot\xi)-2(\tr
k)\bigg\langle\xi,\sum_{j,k}e_k\xi_jk_{kj}\bigg\rangle+2\bigg\langle\xi,\sum_{jk}e_k(k\cdot\xi)_jk_{kj}\bigg\rangle\bigg\}\\
&\quad=\frac{\dim E}{4}\bigg\{|\xi|^2(\tr
k)^2+|k\cdot\xi|^2+2|k\cdot\xi|^2-|\xi|^2|k|^2-2(\tr k)(\xi\cdot
k\cdot\xi)\\
&\qquad\qquad\qquad\quad-2(\tr k)(\xi\cdot
k\cdot\xi)+2|k\cdot\xi|^2\bigg\}\\
&\quad=\frac{\dim E}{4}\bigg\{|\xi|^2(\tr
k)^2+5|k\cdot\xi|^2-|\xi|^2|k|^2-4(\tr k)(\xi\cdot k\cdot\xi) \bigg\}.
\end{align*}
Now we calculate the term
\begin{align*}
\tr\big(\sigma^{(1)}\big)^2&=\frac{1}{4}\tr\bigg\{(\tr k)^2|\xi|^4+(\xi\cdot
k\cdot\xi)^2+\xi(k\cdot\xi)\xi(k\cdot\xi)-2|\xi|^2(\tr k)(\xi\cdot
k\cdot\xi)\\
&\qquad\qquad+2|\xi|^2(\tr k)\xi(k\cdot\xi)-2(\xi\cdot k\cdot\xi)\xi(k\cdot\xi)
\bigg\}\\
&=\frac{\dim E}{4}\bigg\{(\tr k)^2|\xi|^4+(\xi\cdot
k\cdot\xi)^2+2(\xi\cdot k\cdot\xi)^2-|\xi|^2|k\cdot\xi|^2\\
&\qquad\qquad\qquad-2|\xi|^2(\tr k)(\xi\cdot k\cdot\xi)-2|\xi|^2(\tr
k)(\xi\cdot k\cdot\xi)+2(\xi\cdot k\cdot\xi)^2\bigg\}\\
&=\frac{\dim E}{4}\bigg\{(\tr k)^2|\xi|^4+5(\xi\cdot k\cdot\xi)^2-|\xi|^2|k\cdot\xi|^2-4|\xi|^2(\tr
k)(\xi\cdot k\cdot\xi)\bigg\}.
\end{align*}
Adding up the contributions we finally get
\[
\big\langle k,u_s^{(2)}(x,\xi)\,k\big\rangle=
\Big(2S-1\Big)|\xi|^{n-2s-4}\frac{\dim\Sigma
  M}{4}\bigg\{|\xi|^2|k|^2-6|\xi|^2|k\cdot\xi|^2+5(\xi\cdot k\cdot\xi)^2
\bigg\}.
\]
To find $u_s^{(3)}$ we calculate
\begin{align*}
&\tr\sigma^{(0)}=\frac{1}{4}\bigg\{\tr(\xi\xi)-\tr(\xi(k\cdot\xi))\bigg\}\\
&\qquad\quad=\frac{\dim\Sigma
M}{4}\bigg\{(\xi\cdot k\cdot\xi)-|\xi|^2(\tr k)\bigg\},\\
&\tr\sigma^{(1)}=-\frac{\dim E}{2}\bigg\{|\xi|^2(\tr k)-2(\xi\cdot k\cdot\xi)\bigg\}.
\end{align*}
Another ingredient of $u_s^{(3)}$ is
\begin{align*}
&\tr\Big(\sum_{i,j}\sigma_{ij}^{(1)}\big[\xi_i(k\cdot\xi)_j+\xi_j(k\cdot\xi)_i\big]\Big)\\
&\qquad=-\frac{1}{2}\tr\bigg\{\sum_{i,j}\Big[(\tr
k)\deltaij-k_{ij}+e_i\sum_ke_kk_{kj}\Big]\Big(\xi_i(k\cdot\xi)_j+\xi_j(k\cdot\xi)_i\Big)\bigg\}\\
&\qquad=-\frac{1}{2}\tr\bigg\{2(\tr k)(\xi\cdot k\cdot\xi)-2|k\cdot\xi|^2+\xi\sum_{j,k}e_kk_{kj}(k\cdot\xi)_j+(k\cdot\xi)\sum_{j,k}e_k\xi_jk_{kj}\bigg\}\\
&\qquad=-\frac{\dim E}{2}\bigg\{2(\tr k)(\xi\cdot
k\cdot\xi)-2|k\cdot\xi|^2-2|k\cdot\xi|^2\bigg\}\\
&\qquad=\dim E\bigg\{2|k\cdot\xi|^2-(\tr k)(\xi\cdot k\cdot\xi)\bigg\}.
\end{align*}
Thus collecting terms we find
\begin{align*}
&\big\langle
k,u_s^{(3)}(x,\xi)\,k\big\rangle\\
&\qquad=\Big(S-\frac{1}{2}\Big)|\xi|^{n-2s-4}\dim
E\bigg\{4|\xi|^2|k\cdot\xi|^2-3(\xi\cdot k\cdot\xi)^2-|\xi|^2(\tr k)(\xi\cdot k\cdot\xi)\bigg\}.
\end{align*}
Summing up the contributing terms $u_s^{(k)}$ the leading symbol written as $\langle k,u_s(x,\xi)\,k\rangle_g$ equals
\begin{align*}
2^{\lfloor\frac{n}{2}\rfloor-2}\Big(\frac{1}{4\pi}\Big)^{\fracsm{n}{2}}\frac{\Gamma(-S+1)^2}{\Gamma(-2S+2)}|\xi|^{n-2s}
\Bigg\{\Big[2s-(n-1)\Big]\tr\big(K_g\Pi^{\bot}_\xi\big)^2
+\big(\tr{K_g\Pi^{\bot}_\xi\big)^2}\Bigg\},
\end{align*}
which completes the proof of Theorem \ref{DiracHess}.
\end{proof}

\section{Gauge breaking and factorization of stability operators of spectral invariants}\label{projections}
To find the leading symbol by differentiation, we need to pass back to the
unmodified zeta function $\zeta(s)$ and summarize this transition in a lemma. We introduce the notation
\[
\eta_k:=2\sum_{j=1}^{2k+1}\frac{1}{j}-\sum_{j=1}^k\frac{1}{j},
\]
and extract for convenience the $\Gamma$-factor for $\Real s<n/2-1$, writing
\[
\Hess\mathcal{Z}(s)=\frac{\Gamma(-S+1)^2}{\Gamma(-2S+2)}W(s).
\]
Then we have the following lemma.
\begin{lemma}\label{gammalemma}
\begin{itemize}
\item[]
\item[(1)]If $n=2k+1$ odd, then $\Hess\zeta'(0)=\frac{(-1)^{k+1}\pi^{3/2}}{2^{2k+2}(k+1)!}W(0)$.

\item[(2)] If $n=2k$ even, then $\Hess\zeta'(0)=\frac{(-1)^kk!}{(2k+1)!}\Big\{W'(0)+\eta_kW(0)\Big\}$.
\end{itemize}
\end{lemma}

To apply Lemma \ref{gammalemma} we note that in Theorem \ref{DiracHess}
\[
\sigma_L\big(W(s)\big)=2^{\lfloor\frac{n}{2}\rfloor-2}\Big(\frac{1}{4\pi}\Big)^{\fracsm{n}{2}}|\xi|^{n-2s}
\Bigg\{\Big[2s-(n-1)\Big]\tr\big(K_g\Pi^{\bot}_\xi\big)^2
+\big(\tr{K_g\Pi^{\bot}_\xi\big)^2}\Bigg\}
\]
Using this and polarization, we find that
\newline\noindent\newline
$\bullet$\; If $n=2j+1$ is odd then
\begin{align*}
&\sigma_L\big[\Hess_g\zeta'(0)\big](x,\xi)K_g=\\
&\quad\quad\frac{j}{2^{3j+4}\pi^{j-1}(j+1)!}(-1)^j|\xi|^n\bigg\{\Pi^{\bot}_\xi K_g\Pi^{\bot}_\xi
-\frac{1}{n-1}\tr\big(\Pi_\xi^\perp K_g\big)\Pi_\xi^\perp\bigg\}.
\end{align*}
$\bullet$\; If $n=2j$ is even then we have
\begin{align*}
&\sigma_L\big[\Hess_g\zeta'(0)\big](x,\xi)K_g=\frac{j!}{2(2\pi)^j(2j+1)!}\times\\
&(-1)^j|\xi|^n\bigg[\Pi^\bot_\xi K_g\Pi^\perp_\xi+(n-1)\Big\{\Pi^{\bot}_\xi K_g\Pi^{\bot}_\xi
-\frac{1}{n-1}\tr\big(\Pi_\xi^\perp K_g\big)\Pi_\xi^\perp\Big\}\Big(\log|\xi|-\frac{\eta_j}{2}\Big)\bigg].
\end{align*}
\begin{proposition}\emph{(Factorizing out projections)}\newline\label{factorizing}
\noindent{}We can factor out the projections on invariant directions as follows ($n\geq 3$).
\begin{itemize}
\item[$\bullet$] For $n=2j+1$ odd there are $H_{2j+1}\in\mathrm{S}^{n}(M,S^2M)$ and $C(n)>0$ s.t.
\[
\Hess\zeta'(0)=(-1)^j\projconfdiffperp H\projconfdiffperp
\]
\item[$\bullet$]For $n=2j$ even there exists $H_{2j}\in\mathrm{CL}^{n,1}(M,S^2M)$ and $C(n)>0$ s.t.
\[
\Hess\zeta'(0)=(-1)^j\projdiffperp H\projdiffperp.
\]
The leading symbols are respectively
\begin{align}
&\sigma_L^n(H_{2j+1})(x,\xi)=C(n)|\xi|^n,\\
&\sigma_L^n(H_{2j})(x,\xi)=C(n)|\xi|^n\Big[I+(n-1)\Big(\log|\xi|-\frac{\eta_j}{2}\Big)\Phi(x,\xi)\Big],\label{Htwoj}
\end{align}
where $\Phi(x,\xi):S^2_xM\to S^2_xM$ is the orthogonal projection map
\[
\Phi(x,\xi)K_g=\Pi^{\bot}_\xi K_g\Pi^{\bot}_\xi
-\frac{1}{n-1}\tr\big(\Pi^\perp_\xi K_g\big)\Pi_\xi^\perp
\]
\end{itemize}
These factorizations are not unique, since addition of invariant expressions remains undetectable. In addition each $H$ can be chosen symmetric with respect to the $L^2$-inner product.
\end{proposition}
\begin{proof}
From the above expressions for the leading symbols of the stability operators (i.e. $L^2$-Hessians), and of the projections in Proposition \ref{projectionsymbols}, the proposition is true on the leading symbol level, writing
\[
\Hess\zeta'(0)=\Pi_{V^\perp}L\:\Pi_{V^\perp}+R_{-1},
\]
for $V$ either $\confdiff$ or $\diff$, and where $L$ has the leading symbol that remains after factorizing out the leading symbols of the projections. Finally $R_{-1}$ is the remainder term of order $n-1$ or $(n-1,1)$ respectively. Note that in the even-dimensional case, we can factor out $\projdiffperp$ by using
\[
\confdiffperp\subseteq\diffperp.
\]
Now the full stability operator has the corresponding invariant directions, i.e.
\[
\begin{split}
\Hess\zeta'(0)&=\Pi_{V^\bot}\Hess\zeta'(0)\:\Pi_{V^\perp}\\
&=\big(\Pi_{V^\perp}\big)^2L\big(\Pi_{V^\perp}\big)^2+\Pi_{V^\perp}R_{-1}\Pi_{V^\perp}\\
&=\Pi_{V^\perp}\big(L+R_{-1}\big)\Pi_{V^\perp}.
\end{split}
\]
Hence by defining $H=L+R_{-1}$, which leaves the leading symbol unaltered, the factorization also includes the remainder term and we have
\beq\label{FactorEq}
\Hess F = \Pi H \Pi.
\eeq

Using the symmetry of the Hessian and the projections $\Pi$, we see that averaging in (\ref{FactorEq}) with the formal adjoint $\tilde{H}=\frac{1}{2}(H^*+H)$ gives a symmetric choice of $H$ with the desired properties.
\newline\indent{}
\end{proof}

Next is to show that these leading symbols are in fact both positive and hypoelliptic. As mentioned in the introduction, the fact that in even dimensions the term of highest log-degree is singular (namely contains the projection $\Phi$) complicates the construction of the parametrix. As the following proposition shows, it can be done via slightly refined estimates, using the explicit symbol structure of the stability operator.

\begin{proposition}\label{hypoandpositive}
The symbols from Proposition \ref{factorizing} satisfy
\[
\begin{split}
H_{2j+1}&\in\mathrm{HS}_+^{n}(M,S^2M),\\
H_{2j}&\in\mathrm{HCL}_+^{n,1}(M,S^2M).
\end{split}
\]
\end{proposition}
\begin{proof}
The statement concerning $H_{2j+1}$ is immediate. For the positivity of $\sigma_L(H_{2j})$ for large $|\xi|$, it follows from (\ref{Htwoj}), since $\Phi$ is a projection. Equivalently,
\[
(n-1)\tr\big(K_g\Pi^{\bot}_\xi\big)^2\geq\big(\tr{K_g\Pi^{\bot}_\xi\big)^2},
\]
which is nothing but Cauchy-Schwarz' inequality.

It is needed to verify the estimates in Definition \ref{defhypo}, and here Property (1) is immediate.

For (2) we let $\varepsilon$ with $0<\varepsilon<1$ be arbitrary and show that it belongs to the hypoelliptic class of bi-degree $(n+\varepsilon,n)$. We shall apply the following general formula, valid for any orthogonal projection $\Pi$ and $\alpha\in\Reals\backslash\{-1\}$.
\beq\label{projformula}
\big(I+\alpha\Pi\big)^{-1}=I-\frac{\alpha}{\alpha+1}\Pi=I-\Pi-\frac{1}{1+\alpha}\Pi.
\eeq
We need only to prove Property (2) in Definition \ref{defhypo} for the symbol $\sigma$, defined by
\beq
\sigma(x,\xi)=\frac{\sigma_L(H_{2j})(x,\xi)}{C(n)|\xi|^n}.
\eeq
Equation (\ref{projformula}) gives the following formula
\beq
\sigma^{-1}=\big(I-\Phi\big)+\frac{1}{1+(n-1)\Big(\log|\xi|-\frac{\eta_j}{2}\Big)}\Phi.
\eeq
Having written $\sigma^{-1}$ as a sum of the projection onto the orthogonal complement and a term with $1/\log|\xi|$-decay for large $|\xi|$, we get the following estimate.
\beq
\big\vert\sigma^{-1}\Phi\big\vert\leq C\frac{1}{\log|\xi|},\quad |\xi|\quad\text{large}.
\eeq
By differentiating the identity
\beq
\Phi=\Phi^2,
\eeq
using Leibniz' rule, and that $\Phi$ is homogeneous of degree $0$ in $|\xi|$, i.e.
\beq
\big\vert\partial_\xi^\alpha\partial_x^\beta\Phi\big\vert\leq C_{\alpha,\beta}(1+|\xi|)^{-\alpha},
\eeq
we inductively get estimates on all derivatives of $\Phi$ as follows
\beq
\big\vert\sigma^{-1}\big[\partial_\xi^\alpha\partial_x^\beta\Phi\big]\big\vert\leq C_{\alpha,\beta}\frac{(1+|\xi|)^{-\alpha}}{\log|\xi|},\quad |\xi|\quad\text{large}.
\eeq
With these estimates, the symbol
\[
\sigma=I+(n-1)\Big(\log|\xi|-\frac{\eta_j}{2}\Big)\Phi(x,\xi)
\]
is easily seen to satisfy Property (2) in Definition \ref{defhypo} as claimed.
\end{proof}

\section{Completion of the proof of Theorem \ref{DiracExtremals}}
\begin{proof}[Proof of Theorem \ref{DiracExtremals}]
The proof of the main Theorem \ref{DiracExtremals}, giving the generic extremal behavior of the determinant of the Dirac Laplacian $\det{\Dirac^2}$, can now finally be carried out as explained in the introduction, using the same strategy as in the proof of Corollary \ref{ZetaItself}.

Namely, by Proposition \ref{factorizing} the stability operator $\Hess\zeta'(0)$ factorizes in both even and odd dimensions, into a product of projections onto gauge invariance directions, and operators $H_n$ in dimension $n$. By Proposition \ref{hypoandpositive}, these modified stability operators $H_{2j+1}$ and $H_{2j}$ belong to the spaces $\mathrm{HS}_+^{n}(M,S^2M)$ and $\mathrm{HCL}_+^{n,1}(M,S^2M)$, respectively. Thus we may apply the spectral results on semi-bounded\-ness and pure eigenvalue spectrum, tending to infinity and of finite multiplicity, as proven in the main spectral result, Theorem \ref{spectrum}. Again, as in Equation (\ref{Vdef}), one defines the finite-dimensional subspace as the finite direct sum of finite-dimensional eigenspaces for eigenvalues of the sign opposite to that of the leading symbol. In other words the Morse index of the determinant functional is finite at critical points (under the assumptions of Theorem \ref{DiracExtremals}). Arguing again analogously to Equation (\ref{Hneg}) in the proof of Corollay \ref{ZetaItself}, this completes the proof of Theorem \ref{DiracExtremals}.
\end{proof}

\appendix
\section{The Bourguignon-Gauduchon gauge transform of the Atiyah-Singer-Dirac operator}\label{BGAppendix}
For the convenience of the reader, we include a more detailed summary of the paper \cite{BG} on the Bourguignon-Gauduchon isometry and its use in finding the variation of the Atiyah-Singer-Dirac operator. This is essentially a commented translated excerpt of the paper \cite{BG} (see also \cite{Bo} and \cite{Ma}).

Let $V$ be an $n$-dimensional vector space and denote by $\mathscr{M}V$
the convex cone of metrics (i.e. inner products) on $V$. If $g$ is a metric on $V$ we denote by
$\mathcal{F}_OV_g\subseteq\mathcal{F}V$ the space of $g$-orthonormal bases, which has a right
action of $\Oof{n}$. All these actions are free and transitive. Note that if $f\in\mathcal{F}_OV_g$, then
\[
g=(f^{-1})^*e,
\]
where $e$ is the standard metric of $\Reals^n$.
\newline\indent
  Given two metrics $g$ and $h$ we get a $g$-symmetric and positive automorphism
$H_g$ of $V$ by duality:
\beq
h(u,v)=g(H_g(u),v),\quad u,v\in V
\eeq
Note that this is just the usual index-raising from
Riemannian geometry. $H_g^{-\frac{1}{2}}$ is, by its
very definition, an isometry of inner product spaces.
\begin{proposition}[\cite{BG}]
Let $g$ and $h$ be metrics on $V$. Then the map
\[
b_h^g:\mathcal{F}_OV_g\to\mathcal{F}_OV_h
\]
given by $b_h^g(f)=H_g^{-1/2}\circ f$, using the positive square root
of the symmetric positive endomorphism $H_g$, is natural in the sense that:
\begin{itemize}
\item[(1)] $b_h^g=(b_g^h)^{-1}$, $b_g^g=\Id$,
\item[(2)] $b_g^h$ commutes with the right action of $\Oof{n}$ on
  $\mathcal{F}V$.
\item[(3)] If $t\mapsto g_t$ is a smooth curve from $g_0$ to $g_t$, then $b_{g_t}^{g_0}$ gives an isotopy from $\mathcal{F}_OV_{g_0}$ to $\mathcal{F}_OV_{g_t}$.
\end{itemize}
\end{proposition}

It so happens that there is a more geometric way of viewing the
map $b_h^g$. For this, remember that $\mathcal{F}V$ is a principal right
$\Oof{n}$-bundle with the bundle projection
\[
\xymatrix{
\mathcal{F}V\ar[d]^{p}\\
\mathscr{M}V
}
\]
defined by $p(f)=(f^{-1})^*e$. By the polar decomposition of
$\mathrm{Gl}(n)$, this bundle is globally trivial. The fiber
of $g$ is $\mathcal{F}_OV_g$ and the tangent space $T_f\mathcal{F}V$ may be
identified with $L(\Reals^n,V)$. Let $\mathscr{A}_fV$ and
$\mathscr{S}_fV$ be the subspaces of $L(V,\Reals^n)$ which
respectively have anti-symmetric and symmetric matrices with respect
to $f$. The subspace $\mathscr{A}_fV$ is the vertical subspace of the
bundle and $\mathscr{S}_fV$ is a natural complement. The distribution
of subspaces $f\mapsto\mathscr{S}_fV$ is $\Oof{n}$-equivariant, since
symmetry of matrices is preserved by orthogonal conjugation. Thus it
is the horizontal distribution of a certain $\Oof{n}$-connection in
the principal bundle, which we will call the \emph{natural connection}
on $(\mathcal{F}V,p)$.

\begin{proposition}\cite{BG}
Let $g,h\in\mathscr{M}V$. The map $b_h^g$ coincides with the
horizontal transport in $\mathcal{F}V$ with respect to the natural
connection, along the curve $t\mapsto (1-t)g+h$ joining $h$ and $g$ inside $\mathscr{M}V$.
\end{proposition}

This more geometric version extends directly to the spin case. For
this we let $\widetilde{\mathcal{F}}V$ be a realization of the
universal (double) cover of $\mathcal{F}V$. Every fiber
$\mathcal{F}_OV_g$ is covered non-trivially by a manifold
$\widetilde{\mathcal{F}}_OV_g$ diffeomorphic to $\Pin{n}$ and the elements are
called the spinorial bases of $V$ relative to $g$ and the covering
$\mathcal{\widetilde{F}}V$.

To see that there is in fact a diffeomorphism, one may apply polar decomposition in $\Gl{n}$. This gives a deformation retract of $\mathcal{F}V$ onto $\mathcal{F}_OV_g$, which by algebraic topology lifts to a deformation retract of $\mathcal{\widetilde{F}}V$ onto the space
\[
\widetilde{\mathcal{F}}_OV_g:=\pi^{-1}(\mathcal{F}_OV_g),
\]
which is therefore simply connected and is thus the universal covering.

\begin{proposition}[\cite{BG}]
The natural transformation $b$, which to each pair of metrics $g$ and
$h$ associates the diffeomorphism $b_h^g$, lifts to a natural
transformation $\beta$ of the spinorial bases
$\widetilde{\mathcal{F}}V$, which to each pair of metrics $g$ and $h$
associates a $\Pin{n}$-equivariant diffeomorphism from
$\widetilde{\mathcal{F}}_OV_g$ to $\widetilde{\mathcal{F}}_OV_h$.
\end{proposition}

The preceding extends to the case of a Riemannian manifold $M$, and we get a $\Spin{n}$-equivariant bundle map
\[
\beta_\gamma^\eta: \mathcal{P}_\mathrm{Spin}M_\gamma\to\mathcal{P}_\mathrm{Spin}M_\eta.
\]
Here $\gamma$ and $\eta$ are spin metrics that both correspond to the same topological spin
structure. See \cite{BG} and Milnor \cite{Mi}; one considers the group $\mathrm{Gl}^+(n)$ of
matrices with positive determinant and realize a spin structure is as a principal $\widetilde{\mathrm{G}}\mathrm{l}^+(n)$-bundle, denoted $\widetilde{\mathcal{F}}^+M$, which is a double cover as $G$-bundles of $\mathcal{F}^+M$, the positively oriented frames on $M$.

The spinor bundles obtained by associating are denoted by subscript $\gamma$ as
\[
\Sigma M_\gamma=\mathcal{P}_\mathrm{Spin}M_\gamma\times_\rho\Complex^{2^k}.
\]
Again, the $\Spin{n}$-equivariance ensures that
\beq\label{betadef}
\beta_\eta^\gamma([\varphi,v]):=[\beta_\eta^\gamma(\varphi),v]
\eeq
is well defined, so that $\beta$ is a bundle map between spinor bundles
\[
\beta_\eta^\gamma:\Sigma M_\gamma\to\Sigma M_\eta,
\]
if both spin metrics $\gamma$ and $\eta$ correspond to the same
topological spin structure. When this is the case it induces a map on smooth sections (i.e. spinor fields), still denoted in the same way:
\[
\beta_\eta^\gamma:\Cinf{\Sigma M_\gamma}\to\Cinf{\Sigma M_\eta}.
\]
Note also that by the very definition of the Hermitian structure and
of $\beta$ in (\ref{betadef}), $\beta$ is an isometry of Hermitian
vector bundles. Another important property is that $b$ and $\beta$ are
compatible with Clifford multiplication, in the sense that
\beq
\beta_\eta^\gamma(X\cdot_\gamma\varphi)=b_h^g(X)\cdot_\eta\beta_\eta^{\gamma}(\varphi).
\eeq

\subsection{Variation of the Dirac operator and of its eigenvalues}
The gauge transformed Dirac operator may now be described. Fixing a
topological spin structure and spin metrics $\gamma$ and $\eta$
corresponding to this and the metrics $g$ and $h$ respectively, we let
\begin{equation}
\Diracpresuffix{\gamma}{\eta}=\big(\beta_\eta^\gamma)^{-1}\circ\Dirac^\eta\circ \beta_\eta^\gamma.
\end{equation}
Note that this operator
\[
\Diracpresuffix{\gamma}{\eta}:\Cinf{\Sigma M_\gamma}\to\Cinf{\Sigma M_\gamma}
\]
is most definitely \emph{not} the Dirac operator, generically,
corresponding to the spinor metric $\gamma$. Rather the important
point is to consider it as an operator acting canonically on
$\gamma$-spinors but having the same eigenvalues as the Dirac operator
in the spin-metric $\eta$. As such we shall now derive the
corresponding infinitesimal variation.
\newline\indent{}One expression for the Dirac operator is
\[
\Dirac^\gamma\psi=\sum_{i=1}^ne_{i\cdot\gamma}\widetilde{\nabla}_{e_i}^{\gamma}\psi,\quad\psi\in\Cinf{\Sigma
M_\gamma},
\]
using a $g$-orthonormal frame $(e_i)$ and the spinor connection
corresponding to $\gamma$. Note that we use $\sim$ to denote lifted quantities, i.e. $\widetilde{\nabla}^\gamma$ is the lifted spin connection in $\gamma$.

\begin{theorem}[\cite{BG}]
The transformed \mbox{$\Diracpresuffix{\gamma}{\eta}$} of the Dirac operator acts on
$\psi$ a $\gamma$-spinor field as follows
\begin{equation}\label{Diractrans}
\begin{split}
\Diracpresuffix{\gamma}{\eta}\psi=&\sum_{i=1}^{n}e_{i\cdot\gamma}\widetilde{\nabla}_{H_{g}^{-1/2}(e_i)}^{\gamma}\psi\\&+\frac{1}{4}\sum_{i=1}^{n}e_{i\cdot\gamma}\Big[H_g^{1/2}\circ\Big(\nabla_{H_g^{-1/2}(e_i)}^{g}H_g^{-1/2}+\prefix^{g}{A}_{H_g^{-1/2}(e_i)}^h\circ
H_g^{-1/2}\Big)\Big]_{\cdot\gamma}\psi,
\end{split}
\end{equation}
where $(e_i)$ is a $g$-orthonormal frame and
$\prefix^{g}{A}^h=\nabla^h-\nabla^g$ is the difference of the
Levi-Civita connections for $g$ and $h$.
\end{theorem}
\begin{remark}
The formulation here differs from that in \cite{BG}. This is seemingly just a matter of convention, i.e. here we map $u\otimes v\mapsto uv$, while a definition consistent with \cite{BG} would be to map it to $\frac{1}{2}uv$ instead. Note that this ambiguity disappears in the next theorem, which is the result we will apply in the present paper.
\end{remark}
\begin{proof}
Define the transformed spin connection by
\[
\prefix^{\gamma}{\widetilde{\nabla}}^{\eta}_X=(\beta^\gamma_\eta)^{-1}\circ\widetilde{\nabla}^{\eta}_{b_g^h(X)}\circ\beta_\eta^\gamma,\quad X\in\Cinf{TM}.
\]
As might have been expected this is the lift of the transform of the Levi-Civita connection,
defined as
\[
\prefix^{g}{\nabla}^{h}_X=(b^g_h)^{-1}\circ\nabla^{h}_{b_g^h(X)}\circ b_h^g,
\]
which is a $g$-metric connection (by the isometry property), but not
necessarily torsion free. By using the transformed spin connection, however, we get
\[
\Diracpresuffix{\gamma}{\eta}\psi=\big(\beta_\eta^\gamma)^{-1}\bigg(\sum_{i=1}^n b_h^g(e_i)_{\cdot\eta}\widetilde{\nabla}_{e_i}^{\eta}(\beta_\eta^\gamma\psi)\bigg)=\sum_{i=1}^{n}e_{i\cdot\gamma}\prefix^{\gamma}{\widetilde{\nabla}_{H_{g}^{-1/2}(e_i)}}^{\eta}\psi,
\]
by using (\ref{betaCliff}), $b_h^g(e_i)=H_g^{-1/2}(e_i)$ and that by
the isometry $(b_h^g(e_i))$ is an $h$-orthogonal frame, which may be
used for writing the Dirac operator with respect to the spin metric
$\eta$.
\newline\indent{}
Note that in (\ref{Diractrans}) it needs to be explained why and how the
expressions
\[
H^{1/2}\circ\Big(\nabla_X^gH^{-1/2}+\prefix^{g}{A_X}^{h}\circ H^{-1/2}\Big)
\]
define elements in the Clifford algebra. Note that $\prefix^{g}{A_X}^{h}Y$
is tensorial in $Y$, as opposed to the connections themselves. Every $(p,q)$-tensor field $T\in T^p_qM$ is mapped to a Clifford
section as follows: Raise all indices using the metric and apply the
canonical projection in the definition of the Clifford algebra. Note that this is only
injective if the tensor is anti-symmetric (and this is the case here, as follows from being $g$-metric).
\newline\indent{}
It suffices to prove that
\[
\prefix^{\gamma}{\widetilde{\nabla}}^{\eta}_X-\widetilde{\nabla}^\gamma_X
=\frac{1}{4}\Big[H^{1/2}\circ\Big(\nabla_X^gH^{-1/2}+\prefix^{g}{A_X}^{h}\circ H^{-1/2}\Big)\Big]_{\cdot\gamma}.
\]
Note however (see e.g. \cite{BGV} or \cite{LM}) that for two connections
$\nabla^{(1)}$ and $\nabla^{(2)}$
\[
\nabla^{(1)}_Xe_i-\nabla^{(2)}_Xe_i=\sum_{i=1}^n\Big(\omega_{ij}^{(1)}(X)-\omega_{ij}^{(2)}(X)\Big)e_j,
\]
so that
\begin{align}
\Big[\nabla^{(1)}_X-\nabla^{(2)}_X\Big]_{\cdot\gamma}\psi&=2\sum_{i<j}\Big(\omega_{ij}^{(1)}(X)-\omega_{ij}^{(2)}(X)\Big)e_{i\cdot\gamma}e_{j\cdot\gamma}\psi\\
&=4\Big(\widetilde{\nabla}^{(1)}_X-\widetilde{\nabla}^{(2)}_X\Big)\psi,
\end{align}
since again by \cite{BGV}
\begin{equation}\label{spinnablalocal}
\widetilde{\nabla}_X\psi=d\psi(X)-\frac{1}{2}\sum_{i<j}\omega_{ij}(X)e_ie_j\psi,
\end{equation}
viewing $\psi$ in a trivializing neighborhood as a function
\[
\psi:U\to\Complex^{2^k}.
\]
Thus in the case at hand:
\[
\prefix^{\gamma}{\widetilde{\nabla}}^{\eta}_X-\widetilde{\nabla}^\gamma_X
=\frac{1}{4}\Big(\prefix^{g}{\nabla}^{h}_X-\nabla^{g}_X\Big)_{\cdot\gamma}.
\]
Finally by direct computation
\begin{align*}
&H^{1/2}\circ\Big(\nabla_X^gH^{-1/2}+\prefix^{g}{A_X}^{h}\circ
H^{-1/2}\Big)\\
&=\prefix^{g}{\nabla}^{h}_X+H_g^{1/2}\circ\nabla^g_XH_g^{-1/2}-H_g^{1/2}\circ\nabla_X^g\circ
H_g^{-1/2}\\
&=\prefix^{g}{\nabla}^{h}_X+H_g^{1/2}\circ\nabla^g_X\circ
H_g^{-1/2}-H_g^{1/2}\circ H_g^{-1/2}\circ\nabla^g_X-H_g^{1/2}\circ\nabla_X^g\circ
H_g^{-1/2}\\
&=\prefix^{g}{\nabla}^{h}_X-\nabla^{g}_X.
\end{align*}
The second step used the calculus rule
\beq
\nabla_X\circ A=A\circ\nabla_X+\nabla_XA,
\eeq
where $A\in\Cinf{\Endo{TM}}$ is a smooth section of endomorphisms.
\end{proof}
Let $k\in\Cinf{S^2M}$ be a symmetric tensor field, namely a
tangent vector at $g$ in the space of Riemannian metrics on $M$. We now deform
the metric $g_t$ through a smooth curve of metrics, having $k$ as
derivative at $t=0$, for instance $g_t=g+tk$ (e.g. small $t$, $M$ is compact), and find the variation of
the Dirac operator, still for a
fixed topological spin structure.

Before indulging into the proof itself, we mention the following lemma, concerning the
variation of the Levi-Civita connection, which is easily proved used the Koszul formula in a commuting basis of vector fields.
\begin{lemma}\label{nabladotlemma}
The infinitesimal variation of the Levi-Civita connection
corresponding to $g_t$ is the
3-tensor field (raising one index in the ground metric $g_0=g$)
\beq\label{nabladot}
(\nabla^{g_t})'(X,Y,Z)=\frac{1}{2}\Big[(\nabla^g_X)k(Y,Z)+(\nabla^g_Y)k(X,Z)-(\nabla^g_Z)k(X,Y)\Big].
\eeq
\end{lemma}
Now we can continue the proof of the theorem as follows.
\begin{proof}[Proof of Theorem \ref{varDirac}]
Applying (\ref{Diractrans}) with $H_g=(G_t)_g$ gives
\begin{align*}
&\Diracpresuffix{\gamma}{\gamma_t}\psi=\sum_{i=1}^{n}e_{i\cdot\gamma}\widetilde{\nabla}_{(G_t)_g^{-1/2}(e_i)}^{\gamma}\psi\\&+\frac{1}{4}\sum_{i=1}^{n}e_{i\cdot\gamma}\Big[(G_t)_g^{1/2}\circ\Big(\nabla_{(G_t)_g^{-1/2}(e_i)}^{g}(G_t)_g^{-1/2}+\prefix^{g}{A}_{(G_t)_g^{-1/2}(e_i)}^{g_t}\circ
(G_t)_g^{-1/2}\Big)\Big]_{\cdot\gamma}\psi.
\end{align*}
Everywhere $(e_i)$ denotes a $g$-orthonormal frame, i.e. a fixed one
corresponding to the ground metric. Firstly, we fix the
somewhat abusive notation that if $Q_t$ is any $t$-dependent
object, then $Q_0=Q$ and
\[
({Q}_t)':=\frac{d}{dt}_{\big\vert t=0}Q_t.
\]
To proceed we need a few basic derivatives
\begin{align}
&\Big((G_t)_g^{\pm 1/2}(X)\Big)'=\pm\frac{1}{2} K_g(X),\label{basicone}\\
&\big(\nabla_{X_t}^{g_t}\big)'=\big(\nabla_X^{g_t}\big)'+\nabla_{(X_t)'}^{g}\\
&\big(\nabla_{X_t}^{g}Q_t\big)'=\nabla^{g}_{(X_t)'}Q+\nabla_{X}^{g}(Q_t)'\label{basicthree}
\end{align}
from these and similar, noting also $(G_0)_g^{-1/2}=Id$, we get
\[
\Big(\prefix^{g}{A}_{(G_t)_g^{-1/2}(e_i)}^{g_t}\circ
(G_t)_g^{-1/2}\Big)'=\ldots=\big(\prefix^{g}{A}_{e_i}^{g_t}\big)'=\big(\nabla^{g_t}_{e_i}\big)'.
\]
This last derivative was found in the above Lemma
\ref{nabladotlemma}. We also note that by (\ref{basicone}), (\ref{basicthree}) and $\nabla^g_X\Id=0$, we have
\[
\Big((G_t)_g^{1/2}\circ\Big(\nabla_{(G_t)_g^{-1/2}(X)}^{g}(G_t)_g^{-1/2}\Big)\Big)'(Y,Z)=-\frac{1}{2}\big(\nabla^g_{X}k\big)(Y,Z).
\]
\indent{}The final step is a tensor contraction, which is conveniently carried
out in index notation, using $(e_i)$ and the corresponding fundamental
tensors. Consider the 3-tensor $T(X,Y,Z)=-\nabla_Z^g(X,Y)$ and
calculate, preserving orders of the tensor
products. Here and elsewhere $e_ie_je_k$ means $(e_i)_{\cdot\gamma}(e_j)_{\cdot\gamma}(e_k)_{\cdot\gamma}$.

\begin{align*}
\sum_{i=1}^ne_{i\cdot\gamma}\big[T(e_i,\cdot,\cdot)]_{\cdot\gamma}&=
-\sum_{i,j,k}k_{ij,k}e_ie_je_k=\sum_{i,k}k_{ii,k}e_k-\sum_{\substack{i,j,k\\i\neq j}}k_{ij,k}e_ie_je_k\\
&=\quad\sum_{i,k}k_{ii,k}e_k-0=\big[d(\tr_g{k})\big]_{\cdot\gamma},
\end{align*}
using the Clifford relations for orthonormal bases, as well as the
\emph{symmetry} of the tensor field $k$, i.e. $k_{ij}=k_{ji}$.
\newline\indent{}Letting the 3-tensor $T$ be $T(X,Y,Z)=\nabla_Y^g(X,Z)$
\begin{align*}
\sum_{i=1}^ne_{i\cdot\gamma}\big[T(e_i,\cdot,\cdot)]_{\cdot\gamma}&=
\sum_{i,j,k}k_{ik,j}e_ie_je_k=-\sum_{i,k}k_{ik,i}e_k-\sum_{\substack{i,j,k\\i\neq j}}k_{ik,j}e_je_ie_k\\
&=-\big[\diverg_gk\big]_{\cdot\gamma}-\sum_{\substack{i\neq j\\i\neq
  k}}k_{ik,j}e_je_ie_k+\sum_{i\neq j}k_{ii,j}e_j\\
&=-\big[\diverg_gk\big]_{\cdot\gamma}-\sum_{i\neq
  k}k_{ik,k}e_i-0+\big[d(\tr_gk)\big]_{\cdot\gamma}-\sum_ik_{ii,i}e_i\\
&=\big[d(\tr_g{k})\big]_{\cdot\gamma}-2\big[\diverg_gk\big]_{\cdot\gamma}.
\end{align*}
Adding up these contributions finally proves the theorem.
\end{proof}

\bibliographystyle{amsalpha}

\begin{thebibliography}{99}

\bibitem[BCY]{BCY}
  T. P. Branson, S.-Y. A. Chang, P. C. Yang,
  \emph{Estimates and extremals for zeta function determinants on four-manifolds}, Comm. Math. Phys. \textbf{149} (1992), no. 2, 241--262.

\bibitem[BG]{BG}
  J.-P. Bourguignon and P. Gauduchon,
  \emph{Spineurs, opérateurs de Dirac et variations de métriques}, Comm. Math. Phys. \textbf{144} (1992), no. 3, 581--599.

\bibitem[BGV]{BGV}
  N. Berline, E. Getzler and M. Vergne,
  \emph{Heat kernels and Dirac operators},
  Grundlehren der Mathematischen Wissenschaften \textbf{298},
  Springer-Verlag, Berlin, 1992.

\bibitem[Bl]{Bl}
  D. Bleecker,
  \emph{Critical Riemannian manifolds}, J. Diff. Geom. \textbf{14} (1979), p. 599--608.

\bibitem[Bo]{Bo}
 J.-P. Bourguignon,
 \emph{Spinors, Dirac operators, and changes of metrics}, Proc. Sympos. Pure Math., \textbf{54}, Amer. Math. Soc., 1993.

\bibitem[Br1]{BrFuncDet}
  T.P. Branson,
  \emph{The functional determinant}, Lecture Notes Series \textbf{4}, Seoul National University, Global Analysis Research Center, Seoul, 1993.

\bibitem[Br2]{Br2}
  T.P. Branson,
  \emph{$Q$-curvature and spectral invariants}, Rend. Circ. Mat. Palermo, Suppl. No. \textbf{75} (2005), 11--55.

\bibitem[B\O{}1]{BO1}
  T.P. Branson and B. \O{}rsted,
  \emph{Conformal geometry and global invariants}, Differential Geom. Appl. \textbf{1} (1991), no. 3, 279--308.

\bibitem[B\O{}2]{BO2}
  T. Branson and B. \O{}rsted,
  \emph{Explicit functional determinants in four dimensions}, Proc. Amer. Math. Soc. \textbf{113} (1991), no. 3, 669--682.

\bibitem[Ca]{Ca}
  H.B.G. Casimir,
  \emph{On the attraction between two perfectly conducting plates}, Kon. Ned. Akad. Wetensch. Proc. 51:793 (1948).

\bibitem[Ch]{Ch}
  L. Chaumard,
  \emph{Discrétisation de zeta-déterminants d'opérateurs de Schrödinger sur le tore}, Bull. Soc. Math. France \textbf{134} (2006).

\bibitem[Chi]{Chi}
  L. Chiu,
  \emph{Height of flat tori}, Proc. Amer. Math. Soc., \textbf{125} (1997), 723--730.

\bibitem[CLN]{CLN}
  B. Chow, P. Lu, L. Ni,
  \emph{Hamilton's Ricci flow}, Graduate Studies in Mathematics \textbf{77}, AMS, New York (2006).

\bibitem[CLT]{CLT}
  X. Chen, P. Lu, G. Tian,
  \emph{A note on uniformization of Riemann surfaces by Ricci flow}, Proc. Amer. Math. Soc. 134 (2006), no. 11, 3391--3393.

\bibitem[CQ1]{CQ1}
  S.-Y. A. Chang, J. Qing,
  \emph{The zeta functional determinants on manifolds with boundary. I. The formula}, J. Funct. Anal.  \textbf{147} (1997), no. 2, 327--362.

\bibitem[CQ2]{CQ2}
  S.-Y. A. Chang, J. Qing,
  \emph{The zeta functional determinants on manifolds with boundary. II. Extremal metrics and compactness of isospectral set}, J. Funct. Anal. \textbf{147} (1997), no. 2, 363--399.

\bibitem[CY1]{CY1}
  S.-Y. A. Chang, P.C. Yang,
  \emph{Extremal metrics of zeta function determinants on $4$-manifolds},  Ann. of Math. (2) \textbf{142} (1995), no. 1, 171--212.

\bibitem[Da]{Da}
  M. Dahl,
  \emph{On the space of metrics with invertible Dirac operator}, Comment. Math. Helv. \textbf{83} (2008), no. 2, 451--469.

\bibitem[DC]{DC}
  J.S. Dowker, R. Critchley,
  \emph{Scalar effective Lagrangian in de Sitter space}, Phys. Rev. D \textbf{13} (1976), no. 2, 224--234. 

\bibitem[Eb]{Eb}
  D. Ebin,
  \emph{On the space of Riemannian metrics}, Bull. A.M.S. \textbf{75} (1968), 1001--1003.

\bibitem[FG]{FG}
  D. Freed, D. Groisser,
  \emph{The basic geometry of the manifold of Riemannian metrics and of its quotient by the diffeomorphism group}, Michigan Math. J. \textbf{36} (1989), no. 3, 323--344.

\bibitem[Fr]{Fr}
  D. Fried,
  \emph{Analytic torsion and closed geodesics on hyperbolic manifolds}, Invent. Math. \textbf{84} (1986), no. 3, 523--540.

\bibitem[Gi]{Gi}
  P.B. Gilkey,
  \emph{Invariance theory, the heat equation, and the Atiyah-Singer index theorem}, 2nd edition, Studies in Advanced Mathematics, CRC Press (1995).

\bibitem[GP]{GP}
  R.A. Gover, L.J. Peterson,
  \emph{The ambient obstruction tensor and the conformal deformation complex}, Pacific J. Math. \textbf{226} (2006), no. 2, 309--351.

\bibitem[Gu]{Gu}
  M.J. Gursky,
  \emph{Uniqueness of the functional determinant},  Comm. Math. Phys. \textbf{189} (1997), no. 3, 655--665.

\bibitem[Ham]{Ham}
  R. Hamilton,
  \emph{The inverse function theorem of Nash and Moser}, Bull. Amer. Math. Soc. \textbf{7} (1982), no. 1, 65--222.

\bibitem[Haw]{Haw}
  S.W. Hawking,
  \emph{Zeta function regularization of path integrals in curved space time}, Commun. Math. Phys. \textbf{55} (1977), 133--148.

\bibitem[Hi]{Hi}
  N. Hitchin,
  \emph{Harmonic spinors}, Adv. in Math. \textbf{14} (1974), 1--55.

\bibitem[HP]{HP}
  E. D'Hoker, D. H. Phong,
  \emph{On determinants of Laplacians on Riemann surfaces}, Comm. Math. Phys. \textbf{104} (1986), no. 4, 537--545.

\bibitem[KV]{KV}
  M. Kontsevich, S. Vishik,
  \emph{Geometry of determinants of elliptic operators}, Functional
  analysis on the eve of the 21st century, Vol. 1, 173--197, Progr. Math. \textbf{131} (1995), Birkhäuser, Boston.

\bibitem[Li]{Li}
  A. Lichnerowicz,
  \emph{Spineurs harmoniques}, C. R. Acad. Sci. Paris, \textbf{257} (1963) 7--9.

\bibitem[Le]{Le}
  M. Lesch,
  \emph{On the noncommutative residue for pseudodifferential operators
    with log-polyhomogeneous symbols}, Ann. Global Anal. Geom. \textbf{17} (1999), no. 2, 151--187.

\bibitem[LM]{LM}
  H.B. Lawson, M.-L. Michelsohn,
  \emph{Spin geometry}, Princeton Mathematical Series \textbf{38}, Princeton
  University Press, Princeton, NJ, 1989.

\bibitem[LS]{LS}
  S. Levit, U. Smilansky,
  \emph{A theorem on infinite products of eigenvalues of Sturm-
  Liouville type operators}, Proc. Amer. Math. Soc. \textbf{65}, no. 2, 299--302, 1977.

\bibitem[Ma]{Ma}
  S. Maier,
  \emph{Generic metrics and connections on Spin- and Spin$\sp c$-manifolds}, Comm. Math. Phys. \textbf{188} (1997), no. 2, 407--437

\bibitem[Mi]{Mi}
  J.W. Milnor,
  \emph{Remarks concerning spin manifolds}, Differential and Combinatorial Topology (1965), 55--62, Princeton Univ. Press.

\bibitem[Mo1]{Mo1}
  C. Morpurgo,
  \emph{The logarithmic Hardy-Littlewood-Sobolev inequality and extremals of zeta functions on $S\sp n$}, Geom. Funct. Anal.  \textbf{6} (1996), no. 1, 146--171.

\bibitem[Mo2]{Mo2}
  C. Morpurgo,
  \emph{Sharp trace inequalities for intertwining operators on $S\sp n$ and $\Reals\sp n$}, Int. Math. Res. Notices \textbf{20} (1999), 1101--1117.

\bibitem[Mo3]{Mo3}
  C. Morpurgo,
  \emph{Sharp inequalities for functional integrals and traces of conformally invariant operators}, Duke Math. J. \textbf{114} (2002),  no. 3, 477--553.

\bibitem[M{\o}1]{Mol1}
  N. M. M{\o}ller,
  \emph{Dimensional asymptotics of determinants on $S^n$, and proof of B\"{a}r-Schopka's conjecture}, Math. Ann. \textbf{343} (2009), no. 1, 35--51.

\bibitem[M{\o}2]{Mol2}
  N. M. M{\o}ller,
  \emph{Instability of the half-torsion and tests of the AdS/CFT correspondence}, preprint (2009).

\bibitem[M{\o}3]{Mol3}
  N. M. M{\o}ller,
  \emph{A note on Berger metrics on $S^3$ and the determinant of the Dirac operator}, preprint (2009).

\bibitem[M\O{}]{MoO}
  N. M. M{\o}ller, B. \O{}rsted,
  \emph{Rigidity of conformal functionals on spheres}, 29 pages, arXiv:0902.4067.

\bibitem[MT]{MT}
  J. Morgan, G. Tian,
  \emph{Ricci flow and the Poincare conjecture}, Clay Mathematics Monographs \textbf{3}, Cambridge, MA, 2007.

\bibitem[Ok1]{Ok1}
  K. Okikiolu,
  \emph{The Campbell-Hausdorff theorem for elliptic operators and a related trace formula}, Duke Math. J. \textbf{79} (1995), no. 3, 687--722.

\bibitem[Ok2]{Ok2}
  K. Okikiolu,
  \emph{The multiplicative anomaly for determinants of elliptic operators}, Duke Math. J. \textbf{79} (1995), no. 3, 723--750.

\bibitem[Ok3]{Ok3}
  K. Okikiolu,
  \emph{Critical metrics for the determinant of the Laplacian in odd dimensions}, Ann. of Math., \textbf{153} (2001), no. 2, 471--531.

\bibitem[Ok4]{Ok4}
  K. Okikiolu,
  \emph{Hessians of spectral zeta functions}, Duke Math. J., \textbf{124} (2004), no. 3, 517--570.

\bibitem[Ok5]{Ok5}
  K. Okikiolu,
  \emph{Critical metrics for spectral zeta functions}, unpublished work/private communication.

\bibitem[On]{On}
  E. Onofri,
  \emph{On the positivity of the effective action in a theory of random surfaces}, Comm. Math. Phys. \textbf{86} (1982), no. 3, 321--326.

\bibitem[OPS1]{OPS1}
  B. Osgood, R. Phillips and P. Sarnak,
  \emph{Extremals of determinants of Laplacians}, J. Funct. Anal. \textbf{80} (1988), no. 1, 148--211.

\bibitem[OPS2]{OPS2}
  B. Osgood, R. Phillips and P. Sarnak,
  \emph{Compact isospectral sets of surfaces}, J. Funct. Anal. \textbf{80} (1988), no. 1, 212--234.

\bibitem[OPS3]{OPS3}
  B. Osgood, R. Phillips and P. Sarnak,
  \emph{Moduli space, heights and isospectral sets of plane domains}, Ann. of Math. \textbf{129} (1989), no. 2, 293--362.

\bibitem[OW]{OW}
  K. Okikiolu and C. Wang,
  \emph{Hessian of the zeta function for the Laplacian on forms}, Forum Math. \textbf{17} (2005), no. 1, 105--131.

\bibitem[Po]{Po}
  A.M. Polyakov,
  \emph{Quantum geometry of bosonic strings}, Phys. Lett. B \textbf{103} (1981), no. 3, 207--210.

\bibitem[PR]{PR}
  M. Pollicott, A. C. Rocha, André,
  \emph{A remarkable formula for the determinant of the Laplacian}, Invent. Math. \textbf{130} (1997), no. 2, 399--414.

\bibitem[PS]{PS}
  S. Paycha, S. Scott,
  \emph{A Laurent expansion for regularised integrals of holomorphic symbols}, preprint arXiv:math.AP/0506211.

\bibitem[RS1]{RS1}
  D.B. Ray, I. M. Singer,
  \emph{$R$-torsion and the Laplacian on Riemannian manifolds}, Adv. in Math. \textbf{7} (1971), 145--210.

\bibitem[RS2]{RS2}
  D.B. Ray, I. M. Singer,
  \emph{Analytic torsion for complex manifolds}, Ann. of Math. \textbf{98} (1973), 154--177.

\bibitem[Ri]{Ri}
  K. Richardson,
  \emph{Critical points of the determinant of the Laplace operator}, J. Funct. Anal. \textbf{122} (1994), no. 1, 52--83.

\bibitem[Sa]{Sa}
  P. Sarnak,
  \emph{Determinants of Laplacians}, Comm. Math. Phys. \textbf{110} (1987): 113--120.

\bibitem[Sh]{Sh}
  M.A. Shubin,
  \emph{Pseudodifferential operators and spectral theory}, 2nd edition, Springer-Verlag, Berlin, 2001.

\bibitem[SS]{SS}
  P. Sarnak, A. Str\"o{}mbergsson,
  \emph{Minima of Epstein's zeta function and heights of flat tori}, Invent. Math. \textbf{165} (2006), no. 1, 115--151.

\bibitem[SZ]{SZ}
  M. Spreafico, S. Zerbini,
  \emph{Spectral analysis and zeta determinant on the deformed spheres}, Comm. Math. Phys. \textbf{273} (2007), no. 3, 677--704

  
\end{thebibliography}

\end{document}